\documentclass[11pt,twoside,reqno,centertags,draft]{amsart}
\usepackage{amsfonts}
\usepackage{ulem}
\usepackage{color,enumitem,graphicx}
\usepackage[colorlinks=true,urlcolor=blue,
citecolor=red,linkcolor=blue,linktocpage,pdfpagelabels,
bookmarksnumbered,bookmarksopen]{hyperref}

  \usepackage{amsmath,amsthm,amsfonts,amssymb}
\begin{document}

\title{Normalized solutions and mass concentration for supercritical nonlinear Schr\"{o}dinger equations}
\date{}
\maketitle

\vspace{ -1\baselineskip}

{\small
\begin{center}
{\sc  Jianfu Yang} \\
Department of Mathematics,
Jiangxi Normal University\\
Nanchang, Jiangxi 330022,
P.~R.~China\\
email: Jianfu Yang: jfyang\_2000@yahoo.com\\[10pt]
{\sc  Jinge Yang*} \\
School of Sciences,
Nanchang Institute of Technology\\
Nanchang 330099,
P.~R.~China\\
email: Jinge Yang: jgyang2007@yeah.net\\[10pt]

\end{center}
}

\renewcommand{\thefootnote}{}

\begin{quote}
{\bf Abstract.} In this paper, we deal with the existence and concentration of normalized solutions to the supercritical nonlinear Schr\"{o}dinger equation
\begin{equation*}
\left\{
\begin{array}{l}
-\Delta u + V(x) u = \mu_q u + a|u|^q u \quad {\rm in}\quad \mathbb{R}^2,\\
\int_{\mathbb{R}^2}|u|^2\,dx =1,\\
\end{array}
\right.
\end{equation*}
where $\mu_q$ is the Lagrange multiplier. We show that for $q>2$ close to $2$,  the equation admits two solutions: one is the local minimal solution $u_q$ and
another one is the mountain pass solution $v_q$. Furthermore, we study the limiting behavior of $u_q$ and $v_q$ when $q\to 2_+$. Particularly, we describe precisely the blow-up formation of the excited state $v_q$.
\end{quote}

\newcommand{\N}{\mathbb{N}}
\newcommand{\R}{\mathbb{R}}
\newcommand{\Z}{\mathbb{Z}}

\newcommand{\cA}{{\mathcal A}}
\newcommand{\cB}{{\mathcal B}}
\newcommand{\cC}{{\mathcal C}}
\newcommand{\cD}{{\mathcal D}}
\newcommand{\cE}{{\mathcal E}}
\newcommand{\cF}{{\mathcal F}}
\newcommand{\cG}{{\mathcal G}}
\newcommand{\cH}{{\mathcal H}}
\newcommand{\cI}{{\mathcal I}}
\newcommand{\cJ}{{\mathcal J}}
\newcommand{\cK}{{\mathcal K}}
\newcommand{\cL}{{\mathcal L}}
\newcommand{\cM}{{\mathcal M}}
\newcommand{\cN}{{\mathcal N}}
\newcommand{\cO}{{\mathcal O}}
\newcommand{\cP}{{\mathcal P}}
\newcommand{\cQ}{{\mathcal Q}}
\newcommand{\cR}{{\mathcal R}}
\newcommand{\cS}{{\mathcal S}}
\newcommand{\cT}{{\mathcal T}}
\newcommand{\cU}{{\mathcal U}}
\newcommand{\cV}{{\mathcal V}}
\newcommand{\cW}{{\mathcal W}}
\newcommand{\cX}{{\mathcal X}}
\newcommand{\cY}{{\mathcal Y}}
\newcommand{\cZ}{{\mathcal Z}}

\newcommand{\abs}[1]{\lvert#1\rvert}
\newcommand{\xabs}[1]{\left\lvert#1\right\rvert}
\newcommand{\norm}[1]{\lVert#1\rVert}

\newcommand{\loc}{\mathrm{loc}}
\newcommand{\p}{\partial}
\newcommand{\h}{\hskip 5mm}
\newcommand{\ti}{\widetilde}
\newcommand{\D}{\Delta}
\newcommand{\e}{\epsilon}
\newcommand{\bs}{\backslash}
\newcommand{\ep}{\emptyset}
\newcommand{\su}{\subset}
\newcommand{\ds}{\displaystyle}
\newcommand{\ld}{\lambda}
\newcommand{\vp}{\varphi}
\newcommand{\wpp}{W_0^{1,\ p}(\Omega)}
\newcommand{\ino}{\int_\Omega}
\newcommand{\bo}{\overline{\Omega}}
\newcommand{\ccc}{\cC_0^1(\bo)}
\newcommand{\iii}{\opint_{D_1}D_i}

\theoremstyle{plain}
\newtheorem{Thm}{Theorem}[section]
\newtheorem{Lem}[Thm]{Lemma}
\newtheorem{Def}[Thm]{Definition}
\newtheorem{Cor}[Thm]{Corollary}
\newtheorem{Prop}[Thm]{Proposition}
\newtheorem{Rem}[Thm]{Remark}
\newtheorem{Ex}[Thm]{Example}

\numberwithin{equation}{section}
\newcommand{\meas}{\rm meas}
\newcommand{\ess}{\rm ess} \newcommand{\esssup}{\rm ess\,sup}
\newcommand{\essinf}{\rm ess\,inf} \newcommand{\spann}{\rm span}
\newcommand{\clos}{\rm clos} \newcommand{\opint}{\rm int}
\newcommand{\conv}{\rm conv} \newcommand{\dist}{\rm dist}
\newcommand{\id}{\rm id} \newcommand{\gen}{\rm gen}
\newcommand{\opdiv}{\rm div}

\vskip 0.2cm \arraycolsep1.5pt
\newtheorem{Lemma}{Lemma}[section]
\newtheorem{Theorem}{Theorem}[section]
\newtheorem{Definition}{Definition}[section]
\newtheorem{Proposition}{Proposition}[section]
\newtheorem{Remark}{Remark}[section]
\newtheorem{Corollary}{Corollary}[section]

\section{Introduction}

\bigskip

In this paper, we study the existence and asymptotic behavior of standing waves for the following nonlinear Schr\"{o}dinger equation
\begin{equation}\label{eq:1.1}
i\psi_t(x,t)=-\Delta \psi(x,t)+V(x)\psi(x,t)-a|\psi(x,t)|^q\psi(x,t)\ \  (x,t)\in \mathbb{R}^2\times\mathbb{R}^1,
\end{equation}
where $a>0$, $q>2$, and $V$ is an external potential. The wave function $\psi$ is confined to the mass constraint $\int_{\mathbb{R}^2}|\psi|^2\,dx = 1$.

By a standing wave of \eqref{eq:1.1} we mean a solution of equation \eqref{eq:1.1} with the form $\psi(t,x)=e^{i\omega t}u(x)$.
In particular, the function $u$ satisfies
\begin{equation}\label{eq:1.2}
-\Delta u+(V(x)+\omega)u=a|u|^qu\ \ {\rm in} \ \ \mathbb{R}^2
\end{equation}
and
\begin{equation}\label{eq:1.3a}
\int_{\mathbb{R}^2}|u(x)|^2\,dx=1.
\end{equation}

In the case $q=2$, equation \eqref{eq:1.1} stems from the study of Bose-Einstein condensation. It was derived independently by Gross and Pitaevskii, and it  is the main theoretical tool for investigating nonuniform dilute Bose gases
at low temperatures. Especially, equation \eqref{eq:1.2} is called the Gross-Pitaevskii equation. The constant $a$ is the interaction coupling constant fixed by the $s$-wave
scattering length. The case $a>0$ represents that the force between the atoms in the condensates is attractive, and if $a<0$, the force is repulsive. Bose-Einstein condensates with attractive interactions in two dimensions, are described by the Gross-Pitaevskii (GP) energy functional
\[
E_{a}(u)=\frac{1}{2}\int_{\mathbb{R}^2}\big(|\nabla u(x)|^2+V(x)|u(x)|^2\big)\,dx-\frac{a}{q+2}\int_{\mathbb{R}^2}|u(x)|^{4}\,dx.
\]
The exponent $4$ is critical for the functional $E_{a}(u)$ under the unit mass constraint \eqref{eq:1.3a} in the sense that if we make a transformation $u_\lambda(x) = \lambda \varphi(\lambda x),\,\lambda>0$ for any fixed $\varphi \in H^1(\mathbb{R}^2)$ with $\|\varphi\|_{L^2(\mathbb{R}^2)} =1$ in the energy functional
\begin{equation}\label{eq:1.3}
E_{a,q}(u)=\frac{1}{2}\int_{\mathbb{R}^2}\big(|\nabla u(x)|^2+V(x)|u(x)|^2\big)\,dx-\frac{a}{q+2}\int_{\mathbb{R}^2}|u(x)|^{q+2}\,dx,
\end{equation}
then  $\|u_\lambda\|_{L^2(\mathbb{R}^2)} =1$ and $E_{a,q}(u_\lambda)$ is bounded from below if $q< 2$ and unbounded if $q>2$. We refer the cases $1<q<2$, $q =2$ and  $q>2$ as $L^2$ subcritical,   critical and  supercritical respectively.  Hence, the constrained minimization problem
\begin{equation}\label{eq:1.5}
d_a(q):=\inf_{u\in \mathcal{H}, \ \int_{\mathbb{R}^2}|u|^2\,dx=1}E_{a,q}(u)
\end{equation}
can only be considered for subcritical and critical cases, where $\mathcal{H}$ is defined by
\[
\mathcal{H}:=\Big\{u\in H^1(\mathbb{R}^2):\int_{\mathbb{R}^2}V(x)|u(x)|^2\,dx<\infty\Big\}.
\]

If $q = 2$, in the attractive case, the system of Bose-Einstein condensates  collapses whenever the particle number increases beyond a critical
value; see \cite{DGPS, HMDBB, KMS, SSH} etc. Mathematically, it was proved in \cite{GS} that there exists a threshold value $a^*>0$ such that $d_a(2)$ is achieved if $0<a<a^*$,  and
there is no minimizer for $d_a(2)$ if $a\geq a^*$. The threshold value $a^*$ is determined in terms of the solution of the nonlinear scalar field equation
\begin{equation}\label{eq:1.7}
-\Delta u+u=u^3\ \ {\rm in} \ \ \mathbb{R}^2,\ \ u\in H^1(\mathbb{R}^2).
\end{equation}
It is known from \cite{K} that problem \eqref{eq:1.7} admits a unique positive solution up to translations. Such a solution is radially symmetric and exponentially decaying at infinity, see for instance, \cite{BL}. Denote by $Q$ in the sequel the positive solution of \eqref{eq:1.7}, which is radially symmetric about the
origin. It was found in \cite{GS} that the threshold value $a^*$ is given by
\begin{equation}\label{eq:1.6}
a^*:=\|Q\|_{L^2(\mathbb{R}^2)}^2.
\end{equation}
Furthermore, if $V$ is a trap potential, that is $V(x)=\prod_{i=1}^n|x-x_i|^{p} \ \ {\rm with} \ \ p>1$, it was shown in \cite{GS} that
symmetry breaking occurs in the GP minimizers. For $a <a^*$ close to $a^*$, the GP functional $E_{a,2}$ has  at least $n$ different
non-negative minimizers, each of which concentrates at a specific global minimum point $x_i$.

The similar symmetry breaking phenomenon was considered in the subcritical case, i.e. $0<q<2$, for the functional  $E_{a,q}$ in \cite{GZZ}. When $q$ approaching $2$, the limit behavior  of the minimizer of $E_{a,q}$ constrained by \eqref{eq:1.3a} is described by the unique positive solution $\varphi_q$ of the nonlinear scalar field equation
\begin{equation}\label{eq:1.9}
-\Delta u+\frac{2}{q}u=\frac{2}{q}u^{q+1}, \ \ q>0, \ \ u\in H^1(\mathbb{R}^2).
\end{equation}

In this paper, we consider the existence of solutions for the supercritical problem
\begin{equation}\label{eq:0.1a}
\left\{
\begin{array}{l}
-\Delta u + V(x) u = \mu_q u + a|u|^q u \quad {\rm in}\quad \mathbb{R}^2,\\
\int_{\mathbb{R}^2}|u|^2\,dx =1,\\
\end{array}
\right.
\end{equation}
as well as the asymptotic behavior of solutions. That is, we will study the case $q>2$. In the sequel, $\mu_q$ denotes the Lagrange multiplier.

Although in the supercritical case, there is no minimizer for the minimization problem \eqref{eq:1.5}, or no ground state solution for
the problem
\begin{equation}\label{eq:1.6b}
-\Delta u + V(x) u = \mu_q u + a|u|^q u \quad {\rm in}\quad \mathbb{R}^2,
\end{equation}
we can find critical points of $E_{a,q}$ constrained on the manifold
\begin{equation}\label{eq:1.6c}
S(1)=\{u\in \mathcal{H}: \int_{\mathbb{R}^2} |u|^2\,dx = 1\}.
\end{equation}
Such a critical point is an excited state solution of \eqref{eq:1.6b}. Actually, for the supercritical case, it was revealed in \cite{BV,JE} that the functional $E_{a,q}$ with $V=0$ has a mountain pass geometry on $S(1)$. Based on this observation, a variational method was developed to apply to various problems, see \cite{BS,BJ} etc.  We will look for critical points of $E_{a,q}$ on $S(1)$. As observed, one critical point of $E_{a,q}$ on $S(1)$ can be found as a local minimizer,  and another one can be obtained by a variant mountain pass theorem. In fact, we will show that the functional $E_{a,q}$ has a mountain pass geometry on $S(1)$, and it implies that there is a  $(PS)$ sequence of $E_{a,q}$. In order to bound the $(PS)$ sequence, inspired of \cite{GH} and \cite{JE} we establish a variant mountain pass theorem, in which the $(PS)$ sequence is found close to the Pohozaev manifold, see section 2 for details.

\bigskip

We assume that the potential function $V\in C^1(\mathbb{R}^2)$ satisfies
\begin{description}
	\item [$(V_1)$]  $\lim_{|x|\rightarrow \infty}V(x)=\infty,\ \ \ \inf_{x\in \mathbb{R}^2}V(x)=0.$
	\item [$(V_2)$] $(q-1)V+x\cdot\nabla V\geq -C_1$ and $|x\cdot\nabla V(x)|\leq C_2(V(x)+1)$,
\end{description}
where $C_i>0$, $i=1,2$.

\bigskip

In the sequel, we choose $a\in (0,a^*)$. Denote $A_{k}=\{u\in S(1)|\int_{\mathbb{R}^2}|\nabla u|^2\,dx\leq k\}$. The local minimizer will be found in $A_{k}$. In order to
study the asymptotic behavior of critical points of $E_{a,q}$, the number $k$ needs to be selected carefully. Actually, we set
\begin{equation}\label{eq:1.11}
{\tau_q^2}=\Big(\frac{2a_q^*}{qa}\Big)^{\frac{2}{q-2}},
\end{equation}
where $a_q^*$ is defined later in \eqref{eq:1.6a}. Then we obtain the following existence results.

\begin{Theorem}\label{thm1}
Suppose $V$ satisfies $(V_1)$ and $(V_2)$. There exists an $\varepsilon_0>0$ such that, for any $q\in (2,2+\varepsilon_0)$,
 $E_{a,q}(u)$  admits a local positive minimizer $u_q$ in $A_{\tau_q^2}$, that is
\[
E_{a,q}(u_q)=\inf_{u\in A_{\tau_q^2}}E_{a,q}(u),
\]
and a second positive critical point $v_q$  at the mountain pass level on $S(1)$.
\end{Theorem}

We may verify that the trap potential $V$, which has $n\geq 1$ isolated
minima, and that in their vicinity V behaves like a power of the distance from
these points, satisfies conditions $(V_1)$ and $(V_2)$. Precisely, we assume for $n\geq 1$ that

\bigskip

$(V_e)$  $V(x)=\prod_{i=1}^n|x-x_i|^{p_i} \ \ {\rm with} \ \ p_i\geq1, \,i=1,\cdots,n.$

\bigskip

Hence, we  have in particular the following result.

\begin{Corollary}\label{cor1} Suppose $V$ satisfies condition $(V_e)$. Then, $V$ satisfies $(V_1)$ and $(V_2)$. Consequently, the conclusions in Theorem \ref{thm1} also hold.
\end{Corollary}

Next, we study the asymptotic behavior of the local minimizer $u_q$ and the mountain pass point $v_q$ as $q\rightarrow 2_+$. For the supercritical case, it seems that no works concerning the asymptotic behavior of solutions can be found in the literature. In this paper, we give a precise description of the asymptotic behavior of solutions $u_q$ and $v_q$. We commence with the following result.

\begin{Theorem}\label{thm2} Suppose $V$ satisfies $(V_1)$ and $(V_2)$. There hold

$(i)$ $u_q\rightarrow u_0$ in $\mathcal{H}$ as $q\to 2_+$, where $u_0\in H^1(\mathbb{R}^2)$ is a global minimizer of $d_a(2)$, which is defined in \eqref{eq:1.5};

$(ii)$
\[
\lim_{q\rightarrow 2_+}\|v_q\|_{H^1(\mathbb{R}^2)}=\infty.
\]
\end{Theorem}

\bigskip

From Theorem \ref{thm2}, we see that $v_q$ will possibly blow up due to its $H^1$ norm tends to infinity. This allows us to study further the asymptotic behavior of $v_q$.
\begin{Theorem}\label{thm3} Suppose $0<a<a^*$ and the potential $V$ satisfies $(V_e)$. Then, for any sequence $\{q_k\}$ with $q_k\to  2_+$ as $k\rightarrow \infty$, there exist  a subsequence of $\{q_k\}$, still denoted by $\{q_k\}$, $\{x_{k}\}\subset\mathbb{R}^2$, and $\beta>0$ such that
\begin{equation}\label{eq:1.13}
\frac{1}{\|\nabla v_{q_k}\|_2}v_{q_k}\Big(\frac{x+x_k}{\|\nabla v_{q_k}\|_2}\Big)\to \frac{\beta}{\|Q\|_2}Q(\beta x)
\end{equation}
strongly in $L^2(\mathbb{R}^2)$.
\end{Theorem}

\bigskip

The proof of Theorem \ref{thm3} is delicate. In the proof, we will estimate the energy $E_{a,q}(v_q)$ of $v_q$. To this end, we need carefully to choose a path and estimate the energy on it. Meanwhile, we find that
\[
E_{a,q}(v_q)=\frac{q-2}{2q}\Big(\frac{2a_q^*}{qa}\Big)^{\frac{2}{q-2}}+o(1)
\]
as $q\rightarrow 2_+$. We remark that $E_{a,q}(v_q)\to +\infty$ whenever $q\rightarrow 2_+$ in contrast with the subcritical case, where  with the choice of $a>a^*$ the energy $E_{a,q}(w_q)$ of the minimizer $w_q$
goes to $-\infty$ if $q\to 2_-$, see \cite{GZZ}. Essential difficulties will be encountered in estimating $\int_{\mathbb{R}^2}|\nabla v_q|^2\,dx$ and $\int_{\mathbb{R}^2}V(x)v_q^2\,dx$, which can not be done as simple as the subcritical and critical cases. Moreover, although one expects an estimate for these two terms in the supercritical case similar to that for the subcritical and critical cases, it is not able to carry through. Fortunately, we eventually find a suitable estimate enough to serve our purpose.

Finally, we consider a special case $V(x)=|x|^p$. In this case, we have a better description of the limiting function.

\begin{Corollary}\label{cor2} Suppose $0<a<a^*$ and $V(x)=|x|^{p}$, $p\geq1$. Then, for any sequence $\{q_k\}$, $q_k\to  2_+$ as $k\rightarrow \infty$, there exists a subsequence of $\{q_k\}$, still denoted by $\{q_k\}$, such that
\begin{equation}\label{eq:1.13a}
\frac{1}{\|\nabla v_{q_k}\|_2}v_{q_k}\Big(\frac{x}{\|\nabla v_{q_k}\|_2}\Big)\to \frac{1}{\|Q\|_2}Q(x)
\end{equation}
strongly in $L^2(\mathbb{R}^2)$ and
\[
\lim_{k\to\infty}\tau_{q_k}^{-2}\|\nabla v_{q_k}\|_2^2 = 1.
\]
\end{Corollary}

\bigskip
This paper is organized as follows. In section 2, we collect and prove some relevant results for future reference. Then, in section 3, we establish the existence of critical points of $E_{a,q}$. Finally, we analyze the asymptotic behavior of these critical points in sections 4 and 5.

\bigskip

\section{Preliminaries}

\bigskip

In this section, we collect and prove some relevant results for future reference.

For any $q\geq2$, it is well known that problem  \eqref{eq:1.9} possesses a unique radially symmetric positive solution $\varphi_q$. By Lemma 8.1.2 in \cite{CA}, $\varphi_q$ satisfies
\begin{equation}\label{eq:3.3}
\int_{\mathbb{R}^2}|\nabla \varphi_q|^2\,dx=\int_{\mathbb{R}^2}|\varphi_q|^2\,dx=\frac{2}{q+2}\int_{\mathbb{R}^2}|\varphi_q|^{q+2}\,dx.
\end{equation}
It is known from \cite{BL} that there exist positive constants $\delta$, $C$ and $R_0$, independent of $q>0$, such that for any $|x|\geq R_0$,
\begin{equation}\label{eq:3.4}
|\varphi_q(x)|+|\nabla \varphi_q(x)|\leq Ce^{-\delta |x|}.
\end{equation}

\bigskip

\begin{Lemma}\label{lem:5.1} Let $\varphi_q\geq 0$ be the unique solution of \eqref{eq:1.9} with $2\leq q\leq 3$. Then, $\varphi_q\rightarrow Q$ strongly in $H^1(\mathbb{R}^2)$ as $q\to2_+$ and there exist positive constants $C$ and $\delta$ independent of $q$ such that
\begin{equation}\label{eq:5.1}
\varphi_q(x)\leq Ce^{-\delta |x|}\ \ {\rm for} \ \ x\in\mathbb{R}^2.
\end{equation}
Moreover, $a_q^*\to a^* = \|Q\|_2^2$ as $q\to2_+$, where
\begin{equation}\label{eq:1.6a}
a_q^*=\|\varphi_q\|_2^q.
\end{equation}
\end{Lemma}

\begin{proof}
By the Gagliardo-Nirenberg inequality(see \cite{W}), we have
\[
\int_{\mathbb{R}^2}|u|^{q+2}\,dx\leq \frac{q+2}{2\|\varphi_q\|_2^q}\Big(\int_{\mathbb{R}^2}|\nabla u|^2\,dx\Big)^{\frac{q}{2}}\int_{\mathbb{R}^2}|u|^2\,dx
\]
for any $u\in H^1(\mathbb{R}^2)$, that is,
\[
\|\varphi_q\|_2\leq \Big(\frac{q+2}{2}\Big)^{\frac{1}{q}}\frac{\Big(\int_{\mathbb{R}^2}|\nabla u|^2\,dx\Big)^{\frac{1}{2}}\Big(\int_{\mathbb{R}^2}|u|^2\,dx\Big)^{\frac{1}{q}}}{\Big(\int_{\mathbb{R}^2}|u|^{q+2}\,dx
\Big)^{\frac{1}{q}}}.
\]
Choosing $\psi\in C_c^\infty(\mathbb{R}^2)$ such that $0\leq \psi\leq 1$ and $\psi\neq0$, we obtain the uniform $L^2$ bound of $\varphi_q$ in $q$:
\begin{equation}\label{eq:5.2}
\|\varphi_q\|_2
\leq \bigg(\frac{5}{2}\bigg)^{\frac{1}{2}}\frac{\Big(\int_{\mathbb{R}^2}|\nabla \psi|^2\,dx+1\Big)^{\frac{1}{2}}\Big(\int_{\mathbb{R}^2}|\psi|^2\,dx+1\Big)^{\frac{1}{2}}}
{\min\{\Big(\int_{\mathbb{R}^2}|\psi|^5\,dx\Big)^{\frac{1}{2}},
\Big(\int_{\mathbb{R}^2}|\psi|^5\,dx\Big)^{\frac{1}{3}}\}}.
\end{equation}
Therefore, equation \eqref{eq:3.3} implies that $\varphi_q$ is uniformly bounded in $H^1(\mathbb{R}^2)$ for $2\leq q\leq 3$.
Since $\varphi_q\in H_{rad}^1(\mathbb{R}^2)$ and $H_{rad}^1(\mathbb{R}^2)\hookrightarrow L^p(\mathbb{R}^2)$ is compact for $p>2$, there exists $\varphi\in H_{rad}^1(\mathbb{R}^2)$ such that $\varphi_q\rightharpoonup \varphi$ weakly in $H_{rad}^1(\mathbb{R}^2)$ and $\varphi_q\rightarrow \varphi$ strongly in $L^p(\mathbb{R}^2)$ for $p>2$. Thus, we derive for $q\to 2_+$ that
\begin{equation}\label{eq:5.3}
\begin{split}
 &\ \ \int_{\mathbb{R}^2}|\varphi_q^{q+2}-\varphi^4|\,dx\\
 &\leq\int_{\mathbb{R}^2}|\varphi_q^{q+2}-\varphi^{q+2}|\,dx+\int_{\mathbb{R}^2}|\varphi^{q+2}-\varphi^4|\,dx\\
 &\leq C\Big(\int_{\mathbb{R}^2}\Big(|\varphi_q|^{q+1}+|\varphi|^{q+1}\Big)|\varphi_q-\varphi|\,dx
 +\int_{\mathbb{R}^2}|\varphi^{q+2}-\varphi^4|\,dx\Big)\\
 &\leq C\Big[\Big(\int_{\mathbb{R}^2}|\varphi_q|^{\frac{4(q+1)}{3}}+|\varphi|^{\frac{4(q+1)}{3}}\,dx\Big)^{\frac{3}{4}}
 \Big(\int_{\mathbb{R}^2}|\varphi_q-\varphi|^4\,dx\Big)^{\frac{1}{4}}
 +\int_{\mathbb{R}^2}|\varphi^{q+2}-\varphi^4|\,dx\Big]\\
 &\to0.
 \end{split}
\end{equation}
Observe that $\varphi$ satisfies \eqref{eq:1.7}. By the Pohozaev identity \eqref{eq:3.3}, we have
\begin{equation}\label{eq:3.3b}
\int_{\mathbb{R}^2}|\nabla \varphi|^2\,dx=\int_{\mathbb{R}^2}|\varphi|^2\,dx=\frac{1}{2}\int_{\mathbb{R}^2}|\varphi|^4\,dx.
\end{equation}
Hence, we deduce from \eqref{eq:3.3}, \eqref{eq:5.3} and \eqref{eq:3.3b} that
 \[
 \int_{\mathbb{R}^2}|\nabla \varphi_q|^2\,dx = \frac{2}{q+2}\int_{\mathbb{R}^2}|\varphi_q|^{q+2}\,dx \to  \frac{1}{2}\int_{\mathbb{R}^2}|\varphi|^4\,dx=\int_{\mathbb{R}^2}|\nabla \varphi|^2\,dx,
 \]
 which implies that $\varphi_q\rightarrow \varphi$ in $H^1(\mathbb{R}^2)$. By the uniqueness of positive solutions to \eqref{eq:1.7}, we have $\varphi=Q$. Applying the standard elliptic theory, we may show that $\varphi_q$ is uniformly bounded in $L^\infty(\mathbb{R}^2)$. So by \eqref{eq:3.4}, there exists $C, \delta>0$ independent of $q$ such that
\[
\varphi_q(x)\leq Ce^{-\delta |x|}
\]
for $x\in \mathbb{R}^2$.
\end{proof}

\bigskip

In order to find the critical points of the constrained problem, it needs, among other things, to find a $(PS)$ sequence on the constrained manifold. To bound the $(PS)$ sequence, we need the following variant mountain pass theorem.

\bigskip

Let $\mathcal{E}_{a,q}:\mathcal{H}\times \mathbb{R}\to \mathbb{R}$ be the functional
\begin{equation}
\mathcal{E}_{a,q}(u,s)=E_{a,q}(H(u,s))=\frac{1}{2}e^{2s}\int_{\mathbb{R}^2}[|\nabla u|^2+V(e^{-s}x)|u|^2]\,dx
-\frac{a}{q+2}e^{sq}\int_{\mathbb{R}^2}|u|^{q+2}\,dx,
\end{equation}
where $H(u,s)=e^su(e^sx)$.
For $\varphi_1, \varphi_2\in S(1)$ being nonnegative, we  define
\begin{equation}\label{eq:5.4}
\mathcal{P}_q=\{\gamma\in C([0,1], S(1)\times \mathbb{R}):\gamma(0)=(\varphi_1,0),\
\gamma(1)=(\varphi_2,0)\}
\end{equation}
and
\begin{equation}\label{eq:5.5}
b_q=\inf_{\gamma\in \mathcal{P}_q}\max_{t\in [0,1]}\mathcal{E}_{a,q}(\gamma(t)).
\end{equation}

\bigskip

Let $c_q$ be the mountain pass level for $E_{a,q}$ defined by
\begin{equation}\label{eq:2.18a}
 c_q=\inf_{g\in \Gamma_q}\max_{t\in [0,1]} E_{a,q}(g(t)),
\end{equation}
where
\begin{equation}\label{eq:2.17a}
\Gamma_q=\{g\in C([0,1],S(1))|g(0)=\varphi_1,\ g(1)=\varphi_2\}.
\end{equation}

\bigskip
Since $H(\mathcal{P}_q)\subset\Gamma_q$ and $(g,0)\in \mathcal{P}_q$ for any $g\in \Gamma_q$, we have the following result.
\begin{Lemma}\label{lem:5.2}  There holds  $b_q=c_q$.
\end{Lemma}





\bigskip

Let us denote by $Y$ the space $\mathcal{H}\times \mathbb{R}$ with the norm $\|\cdot\|_Y^2=\|\cdot\|_{\mathcal{H}}^2+\|\cdot\|_\mathbb{R}^2$ and denote by $Y^{-1}$ its dual
space.
\begin{Proposition}\label{prop:2.1aa}
Suppose that
\begin{equation}\label{mpt0}
b_q=\inf_{\gamma\in \mathcal{P}_q}\max_{t\in [0,1]}\mathcal{E}_{a,q}(\gamma(t))>\sup\{\mathcal{E}_{a,q}(\gamma(0)),\ \mathcal{E}_{a,q}(\gamma(1))\}.
\end{equation}
Then, there exist $\{\gamma_n\}:=\{(g_n,0)\}\subset \mathcal{P}_q$ with $g_n\geq 0$ and $\{(w_n,s_n)\}\subset S(1)\times \mathbb{R}$ such that
\begin{equation}\label{eq:5.6}
\lim_{n\rightarrow \infty}\sup_{\gamma_n}\mathcal{E}_{a,q}=  b_q;
\end{equation}
\begin{equation}\label{eq:5.7}
\lim_{n\rightarrow \infty} \mathcal{E}_{a,q}(w_n, s_n)=b_q;
\end{equation}
\begin{equation}\label{eq:5.8}
\|\mathcal{E}_{a,q}'|_{S(1)\times \mathbb{R}}(w_n,s_n)\|_{Y^{-1}}\rightarrow 0;
\end{equation}
\begin{equation}\label{eq:5.9}
\lim_{n\rightarrow \infty} dist\big((w_n,s_n), (g_n,0)\big)=0.
\end{equation}
\end{Proposition}

\begin{proof}  By Lemma 2.3, there exists $h_n\in\Gamma_q$ such that $\lim_{n\to\infty}\sup_{h_n}E_{a,q}(h_n)=c_q=b_q$. Obviously, $|h_n|\in \Gamma_q$ and $c_q\leq E_{a,q}(|h_n|)\leq E_{a,q}(h_n)$. Hence, $\lim_{n\to\infty}\sup_{h_n}E_{a,q}(|h_n|)=b_q$ and
$\lim_{n\to \infty}\sup_{(|h_n|,0)}\mathcal{E}_{a,q}=b_q$.
The conclusion follows from Theorem 3.2 in \cite{GH}, where we choose $\varphi=\mathcal{E}_{a,q}$, $\mathcal{F}=\mathcal{P}_q$, $B=(\gamma(0), \gamma(1))$, $c=b_q$, $X=S(1)\times \mathbb{R}$,  $\gamma_n=(g_n,0)$ with $g_n=|h_n|$.
\end{proof}

\bigskip

Finally, we have the following Pohozaev identity.
\begin{Lemma}\label{lem:5.3} If $u$ solves
\begin{equation}\label{eq:5.12a}
-\Delta u+V(x)u=au^{q+1}+\mu u,
\end{equation}
 then
\begin{equation}\label{eq:5.12}
\begin{split}
\int_{\mathbb{R}^2}|\nabla u|^2\,dx-\frac{1}{2}\int_{\mathbb{R}^2}x\cdot\nabla V|u|^2\,dx=\frac{qa}{q+2}\int_{\mathbb{R}^2}|u|^{q+2}\,dx.
\end{split}
\end{equation}
\end{Lemma}
\begin{proof}
It is known from \cite{AP} that
\begin{equation}\label{eq:5.13}
\begin{split}
\int_{\mathbb{R}^2}V|u|^2\,dx+\frac{1}{2}\int_{\mathbb{R}^2}x\cdot\nabla V|u|^2\,dx
=\frac{2a}{q+2}\int_{\mathbb{R}^2}|u|^{q+2}\,dx+\mu\int_{\mathbb{R}^2}|u|^{2}\,dx.
\end{split}
\end{equation}
Multiplying \eqref{eq:5.12a} by $u$ and integrating by part, we have
\begin{equation}\label{eq:5.14}
\begin{split}
\int_{\mathbb{R}^2}|\nabla u|^2\,dx+\int_{\mathbb{R}^2}V|u|^2\,dx
=a\int_{\mathbb{R}^2}|u|^{q+2}\,dx+\mu\int_{\mathbb{R}^2}|u|^{2}\,dx.
\end{split}
\end{equation}
The Pohozaev identity \eqref{eq:5.12} follows from \eqref{eq:5.13} and \eqref{eq:5.14}.
\end{proof}

\bigskip

\section{Existence}

\bigskip

In this section, we show the existence of two critical points of the functional $E_{a,q}(u)$ on the sphere $S(1)$ defined in \eqref{eq:1.6c}.
The first critical point of $E_{a,q}(u)$ will be found as
a minimizer of the minimization problem
\begin{equation}\label{eq:2.2}
m_k=\inf_{u\in A_k}E_{a,q}(u),
\end{equation}
where
\[
A_k=\{u\in S(1)|\int_{\mathbb{R}^2}|\nabla u|^2\,dx< k\}.
\]
Once $m_k$ is achieved, it is necessary to show that the minimizer is not on the boundary of $A_k$:
\[
\partial A_k=\{u\in S(1)|\int_{\mathbb{R}^2}|\nabla u|^2\,dx=k\}.
\]
The minimizer is then a critical point of $E_{a,q}(u)$. The second critical point of $E_{a,q}(u)$ is obtained by the mountain pass theorem.

At the beginning, we recall the following compactness lemma, which can be proved as that in \cite{YY}.

\begin{Lemma}\label{lem:2.1} Suppose $V\in L_{loc}^\infty(\mathbb{R}^2)$ and  $\lim_{x\to\infty}V(x)=\infty$. Then the embedding $\mathcal{H}\hookrightarrow L^p(\mathbb{R}^2)$ is compact for any $ p\in[2,\infty)$.
\end{Lemma}

\bigskip

Now, we show that there is a minimizer of $m_k$, which is a critical point of $E_{a,q}(u)$. In the sequel, we denote $f^{t}(x)=tf(tx)$ for any function $f$.

\begin{Proposition}\label{prop:2.1} Suppose $(V_1)$ and $(V_2)$.
For each $a\in (0,a^*)$, there exists a positive critical point  $u_q\in A_{\tau_q^2}$ of the functional $E_{a,q}(u)$  such that
\[
E_{a,q}(u_q)= \inf_{u\in A_{\tau_q^2}}E_{a,q}(u)
\]
if $q>2$ close to $2$,  where $\tau_q^2$ is defined in \eqref{eq:1.11}.
\end{Proposition}

\bigskip

\begin{proof}
We consider the minimization problem
\[
m_k=\inf_{u\in A_k}E_{a,q}(u).
\]

Fix $k>0$, we claim that $m_k$ is achieved. Indeed,  let $\{u_n\}\subset A_k$ be a minimizing sequence of $m_k$, which is obviously bounded in $\mathcal{H}$.
We can assume that it converges weakly to $u_k\in \mathcal{H}$. By Lemma \ref{lem:2.1},  we have $u_k\in S(1)$. The lower semi-continuity of the functional $E_{a,q}$ implies
\[
E_{a,q}(u_k)\leq \liminf_{n\rightarrow \infty} E_{a,q}(u_n)=m_k.
\]
Therefore, $E_{a,q}(u_k)=m_k$ and $\|\nabla u_k\|_2^2\leq k$. That is, $u_k$ is a minimizer of $m_k$.

Next,  we show that  $u_k$ is a critical point of $E_{a,q}(u)$.  It is sufficient to prove $u_k\not\in\partial A_k$. Now, we will find a suitable $k>0$ so that $u_k$ belongs to $A_k$.
This will be done if we can find an element $\varphi\in A_k$ so that
\begin{equation}\label{eq:2.6a}
E_{a,q}(\varphi)<\inf_{u\in\partial A_k}E_{a,q}(u).
\end{equation}

In the following, we show that inequality \eqref{eq:2.6a} is valid for $k= \tau^2_q$, where $\tau_q$ is given in \eqref{eq:1.11}.
By the Gagliardo-Nirenberg inequality\cite{W}, we have
\begin{equation}\label{eq:2.4}
\int_{\mathbb{R}^2}|u|^{q+2}\,dx\leq \frac{q+2}{2 a_q^*}\Big(\int_{\mathbb{R}^2}|\nabla u|^2\,dx\Big)^{\frac{q}{2}}\int_{\mathbb{R}^2}u^2\,dx,
\end{equation}
where $a_q^*$ is defined in \eqref{eq:1.6a}.  This implies that for any $u\in S(1)$,
\begin{equation}\label{eq:2.5}
E_{a,q}|_{V=0}(u)\geq \frac{1}{2}\int_{\mathbb{R}^2}|\nabla u|^2\,dx-\frac{a}{2a^*_q}\Big(\int_{\mathbb{R}^2}|\nabla u|^2\,dx\Big)^{\frac{q}{2}},
\end{equation}
where $E_{a,q}|_{V=0}(u)$ denotes the functional obtained by taking $V\equiv 0$ in $E_{a,q}$.
In view of \eqref{eq:2.5}, we consider the function $g: \mathbb{R}\mapsto \mathbb{R}$ defined by
\[
g(s)=\frac{1}{2}s-\frac{a}{2a_q^*}s^{\frac{q}{2}}.
\]
We may verify that $g$ is increasing in $(0,{\tau_q^2})$ and decreasing in $({\tau_q^2},\infty)$.
Therefore, the function $g$ attains its maximum at $s={\tau_q^2}$, and
\begin{equation}\label{eq:2.8}
g({\tau_q^2})=\frac{q-2}{2q}\tau_q^2.
\end{equation}
Since $a<a^*$ and by Lemma \ref{lem:5.1} $a_q^*\rightarrow a^*$ as $q\rightarrow2$, we remark that
\begin{equation*}
\frac{q-2}{2q}\tau_q^2\rightarrow +\infty\ \ {
\rm as}\ \ q\rightarrow2_+.
\end{equation*}
By \eqref{eq:2.5}, we have
\begin{equation}\label{eq:2.9}
\inf_{u\in \partial A_{{\tau_q^2}}}E_{a,q}(u)\geq \inf_{u\in \partial A_{{\tau_q^2}}}E_{a,q}|_{V=0}(u)\geq \frac{q-2}{2q}{\tau_q^2}=:\theta {\tau_q^2}.
\end{equation}
Apparently, if $u\in A_{\frac{1}{4}\theta {\tau_q^2}}$, then
\begin{equation}\label{eq:2.10}
E_{a,q}|_{V=0}(u)\leq \frac{1}{8}\theta {\tau_q^2}.
\end{equation}

Choose $\varphi$ such that
$\varphi\in C_c^\infty(\mathbb{R}^2)$, $\varphi\geq 0,$ and $\|\varphi\|_2^2=1$ and let
\begin{equation}\label{eq:2.10a}
t_0= t_0(q)=\frac{1}{4}\theta {\tau_q^2}\Big(\int_{\mathbb{R}^2}|\nabla \varphi|^2\,dx\Big)^{-\frac{1}{2}}
\end{equation}
be such that
\begin{equation}\label{eq:2.11}
\int_{\mathbb{R}^2}|\nabla \varphi^{t_0}|^2\,dx=\frac{1}{4}\theta {\tau_q^2}.
\end{equation}

Since $\tau_q\rightarrow \infty$ as $q\to 2_+$, so does $t_0$. By Lemma \ref{lem:5.1} and the Lebesgue dominated convergence theorem,
\begin{equation}\label{eq:2.12}
\int_{\mathbb{R}^2}V(x)|\varphi^{t_0}(x)|^2\,dx=\int_{\mathbb{R}^2} V\Big(\frac{x}{t_0}\Big)|\varphi(x)|^2\,dx\rightarrow V(0)
\end{equation}
as $q\to 2_+$.
It follows from \eqref{eq:2.11} and \eqref{eq:2.12} that
\begin{equation*}
E_{a,q}(\varphi^{t_0})\leq \frac{1}{8}\theta {\tau_q^2}+2V(0).
\end{equation*}
We deduce from \eqref{eq:2.9} that
\begin{equation}\label{eq:2.13}
\inf_{u\in \partial A_{{\tau_q^2}}}E_{a,q}(u)-E_{a,q}(\varphi^{t_0})\geq \frac{7}{8}\theta {\tau_q^2}-2V(0)>0
\end{equation}
for $q>2$ and close to $2$. Hence, we have $\varphi^{t_0}\in A_{{\tau_q^2}}$  and
\[
m_{\tau_q^2}\leq E_{a,q}(\varphi^{t_0})<\inf_{u\in \partial A_{{\tau_q^2}}}E_{a,q}(u)
\]
for $q>2$ and close to $2$. Consequently,  $E_{a,q}(u)$ attains its minimum at $u_{{\tau_q^2}}\in A_{{\tau_q^2}}$ for $q$ close to 2.
Note that $E_{a,q}(|u|)\leq E_{a,q}(u)$, we can assume that $u_k$ is  nonnegative. In addition, $u_{{\tau_q^2}}$ solves \eqref{eq:1.6b} for some Lagrange multiplier $\mu_q$. By the strong maximum principle, $u_{{\tau_q^2}}>0$. The proof is complete.
\end{proof}

\bigskip

Once we show that the functional $E_{a,q}(u)$ has a mountain-pass geometry, we may find a $(PS)$ sequence of $E_{a,q}(u)$, which is close to the Pohozaev manifold. Indeed, we have the following lemma, which is motivated by \cite{BL, GH, JE}.
\begin{Lemma}\label{lem:1.1}
Suppose
\begin{equation}\label{mpt1}
 c_q>\max\{E_{a,q}(\varphi_1), E_{a,q}(\varphi_2)\},
\end{equation}
where $c_q$ is defined in \eqref{eq:2.18a}.   Then, there is  a sequence $\{u_n\}\subset S(1)$  such that
\begin{equation}\label{eq:2.14}
\begin{split}
E_{a,q}(u_n)\rightarrow c_q,\\
\|E_{a,q}'|_{S(1)}(u_n)\|_{\mathcal{H}^{-1}}\rightarrow 0,\\
Q_q(u_n)\to 0, \ \ as \ \ n\rightarrow \infty,
\end{split}
\end{equation}
where
\begin{equation}\label{eq:2.15}
Q_q(u)=\int_{\mathbb{R}^2}|\nabla u|^2\,dx-\frac{1}{2}\int_{\mathbb{R}^2}x\cdot\nabla V|u|^2\,dx-\frac{qa}{q+2}\int_{\mathbb{R}^2}|u|^{q+2}\,dx
\end{equation}
and $\mathcal{H}^{-1}$ denotes the dual space of $\mathcal{H}$.

Moreover, there is a sequence $v_n\in S(1)$ with $v_n\geq 0$ such that
\begin{equation}\label{eq:2.15a}
\lim_{n\to\infty}\|u_n-v_n\|_\mathcal{H}=0.\\
\end{equation}
\end{Lemma}

\begin{proof}
By Lemma \ref{lem:5.2}, we have $b_q = c_q$. Hence,  equation \eqref{mpt1} implies that equation \eqref{mpt0} holds true, so do the results in Proposition \ref{prop:2.1}.

Let $u_n=H(w_n, s_n)$. By \eqref{eq:5.7} and Lemma \ref{lem:5.2}, we have
\[
E_{a,q}(u_n)=\mathcal{E}_{a,q}(w_n,s_n)\rightarrow c_q.
\]
By \eqref{eq:5.8},
\begin{equation}\label{eq:5.10}
\langle \mathcal{E}'_{a,q}(w_n,s_n), z\rangle_{Y^{-1}, Y}=o(\|z\|_Y),
\end{equation}
for all $z\in \mathcal{T}_{(w_n,s_n)}:=\{(z_1,z_2)\in Y: \langle w_n, z_1\rangle_2=0\}$.

Choosing $z=(0,1)$ in \eqref{eq:5.10}, we find
\begin{equation*}
\begin{split}
&\langle \mathcal{E}'_{a,q}(w_n, s_n), (0,1)\rangle\\
&=\frac{d}{dt}\mathcal{E}_{a,q}(w_n, s_n+t)|_{t=0}\\
&=\frac{d}{dt}E_{a,q}(H(w_n,s_n+t))|_{t=0}\\
&=\frac{d}{dt}E_{a,q}\big(e^{s_n+t}w_n(e^{s_n+t}x)\big)|_{t=0}\\
&=e^{2s_n}\int_{\mathbb{R}^2}|\nabla w_n|^2\,dx-\frac{1}{2}\int_{\mathbb{R}^2}\nabla V(e^{-s_n}x)\cdot e^{-s_n}x|w_n|^2\,dx-\frac{qa}{q+2}e^{qs_n}\int_{\mathbb{R}^2}|w_n|^{q+2}\,dx\\
&=\int_{\mathbb{R}^2}|\nabla u_n|^2\,dx-\frac{1}{2}\int_{\mathbb{R}^2}x\cdot\nabla V|u_n|^2\,dx-\frac{qa}{q+2}\int_{\mathbb{R}^2}|u_n|^{q+2}\,dx\\
&=Q_q(u_n)\rightarrow 0,
\end{split}
\end{equation*}
as $n\to\infty$.

For any $\varphi\in T_{u_n}:=\{u_n\in \mathcal{H}, \langle u_n, \varphi\rangle_2=0\}$, setting
$\psi=e^{-s_n}\varphi(e^{-s_n}x)$,
\begin{equation}\label{eq:5.11a}
\begin{split}
\langle E_{a,q}'|_{S(1)}(u_n), \varphi\rangle&=\langle E_{a,q}'(u_n), \varphi\rangle\\
&=\frac{d}{dt}E_{a,q}(u_n+t\varphi)|_{t=0}\\
&=\frac{d}{dt}E_{a,q}(e^{s_n}w_n(e^{s_n}x))+te^{s_n}\psi(e^{s_n}x))|_{t=0}\\
&=\frac{d}{dt}E_{a,q}(H(w_n+t\psi, s_n))|_{t=0}\\
&=\frac{d}{dt}\mathcal{E}_{a,q}(w_n+t\psi, s_n)|_{t=0}\\
&=\langle \mathcal{E}_{a,q}'(w_n,s_n),(\psi,0)\rangle.
\end{split}
\end{equation}
Since $\langle w_n,\psi\rangle_2=\langle u_n, \varphi\rangle_2=0$,  $(\psi, 0)\in \mathcal{T}_{(w_n,s_n)}$. By \eqref{eq:5.8} and \eqref{eq:5.11a}, we have
\begin{equation}\label{eq:5.11}
\langle E_{a,q}'|_{S(1)}(u_n), \varphi\rangle=o(\|\psi\|_{\mathcal{H}}),
\end{equation}
and \eqref{eq:5.9} implies
\begin{equation}\label{3.21-a}
|s_n|\leq dist\big((w_n,s_n), (g_n,0)\big)\rightarrow 0
\end{equation}
as $n\to\infty$.  Hence,
$\|\psi\|_{\mathcal{H}}\leq 4\|\varphi\|_{\mathcal{H}}$. It results
$$
\langle E_{a,q}'|_{S(1)}(u_n),\varphi\rangle=o(\|\varphi\|_{\mathcal{H}}).
$$
By \eqref{eq:5.9} and \eqref{3.21-a}, we have $\lim_{n\to\infty}dist_{\mathcal{H}}(w_n, g_n)=0$. So we may choose $t_n\in [0,1]$ such that
\[
\|w_n-g_n(t_n)\|_{\mathcal{H}}=dist_{\mathcal{H}}(w_n,g_n).
\]
Let $v_n=g_n(t_n)\geq 0$. We have
$\lim_{n\to\infty}\|w_n-v_n\|_{\mathcal{H}}=0$. This with \eqref{3.21-a} yields
\[
\lim_{n\to \infty}\|u_n-v_n\|_{\mathcal{H}}=\lim_{n\to \infty}\|H(w_n,s_n)-v_n\|_{\mathcal{H}}=\lim_{n\to \infty}\|w_n-v_n\|_{\mathcal{H}}=0.
\]

\end{proof}

\bigskip

Now, we seek for the second critical point of $E_{a,q}(u)$ by the variant mountain pass theorem.
\begin{Proposition}\label{prop:2.2} Suppose $(V_1)$ and $(V_2)$ hold. If $q>2$ and close to $2$, then $E_{a,q}(u)$ admits a second critical point $v_q$ on $S(1)$ at the mountain pass level.
\end{Proposition}

\begin{proof} First, we verify that $E_{a,q}(u)$ has a mountain pass geometry on $S(1)$.

Let $\varphi\in C_c^\infty(\mathbb{R}^2)$ be such that $\varphi\geq 0,$ and $\|\varphi\|_2^2=1$.
Denote $\varphi^{t}(x)=t\varphi(tx)$. We find
\[
E_{a,q}(\varphi^t)=\frac{1}{2}t^2\int_{\mathbb{R}^2}|\nabla \varphi|^2\,dx+\frac{1}{2}\int_{\mathbb{R}^2}V\Big(\frac{x}{t}\Big)\varphi^2\,dx-
\frac{at^q}{q+2}\int_{\mathbb{R}^2}|\varphi|^{q+2}\,dx.
\]
Since
\begin{equation}\label{eq:2.16}
\int_{\mathbb{R}^2}V\Big(\frac{x}{t}\Big)\varphi^2\,dx\rightarrow V(0)\ \ {\rm as}\ \ t\rightarrow \infty
\end{equation}
and $q>2$, we have
\begin{equation*}
E_{a,q}(\varphi^t)\rightarrow -\infty\ \ {\rm as}\ \ t\rightarrow \infty.
\end{equation*}
 Taking
\begin{equation}\label{t1}
t_1=C2^{\frac{1}{(q-2)^2}}
\end{equation}
 with $C$ large enough and independent of $q$ such that $E_{a,q}(\varphi^{t_1})<0$ for $q>2$ and close to $2$, we define
\begin{equation}\label{eq:2.17}
\Gamma_q=\{g\in C([0,1],S(1))|g(0)=\varphi^{t_0}, g(1)=\varphi^{t_1}\}.
\end{equation}

 By the choice of $t_1$ and \eqref{eq:2.13}, we see that \eqref{mpt1} is valid.
This means that $E_{a,q}$ has the mountain pass geometry. By Lemma \ref{lem:1.1},
 there is a sequence $\{u_n\}\subset S(1)$  such that
\begin{equation}\label{eq:2.19}
\begin{split}
E_{a,q}(u_n)\rightarrow c_q,\\
\|E_{a,q}'|_{S(1)}(u_n)\|_{\mathcal{H}^{-1}}\rightarrow 0,\\
Q_q(u_n)\to0, \ \ as \ \ n\rightarrow \infty.
\end{split}
\end{equation}
Hence, \eqref{eq:2.14}, the  identity
\begin{equation}\label{eq:2.21}
\frac{q-2}{2}\int_{\mathbb{R}^2}|\nabla u|^2\,dx+\frac{1}{2}\int_{\mathbb{R}^2}(qV+x\cdot\nabla V)u^2\,dx=qE_{a,q}(u)-Q_q(u)
\end{equation}
and $(V_2)$ give that the sequence $\{u_n\}$ is bounded in $\mathcal{H}$. So there exists $v_q\in \mathcal{H}$ such that
$u_n\rightharpoonup v_q$ weakly in $\mathcal{H}$. Since $E_{a,q}'|_{S(1)}(u_n)\rightarrow 0$ in $\mathcal{H}^{-1}$, by Lemma 3 in \cite{BL} , there is a $\mu_q^n$ such that
\begin{equation}\label{eq:2.22}
-\Delta u_n+Vu_n-a|u_n|^qu_n-\mu_q^nu_n\rightarrow 0 \quad {\rm in}\quad \mathcal{H}^{-1}.
\end{equation}
Hence,
\begin{equation*}
\mu_q^n=\int_{\mathbb{R}^2}|\nabla u_n|^2\,dx+\int_{\mathbb{R}^2}V|u_n|^2\,dx-a\int_{\mathbb{R}^2}|u_n|^{q+2}\,dx+o(\|u_n\|_{\mathcal{H}}),
\end{equation*}
which is bounded. Without of the loss of generality, we may assume $\mu_q^n\rightarrow \mu_q$ as $n\rightarrow \infty$. It yields
\begin{equation}\label{eq:2.23}
-\Delta v_q+Vv_q-a|v_q|^qv_q-\mu_qv_q=0.
\end{equation}
We deduce from \eqref{eq:2.22} and \eqref{eq:2.23} that
\[
\int_{\mathbb{R}^2}|\nabla u_n|^2\,dx+\int_{\mathbb{R}^2}Vu_n^2\,dx-a\int_{\mathbb{R}^2}|u_n|^{q+2}\,dx-\mu_q^n=o(\|u_n\|_{\mathcal{H}})
\]
and
\[
\int_{\mathbb{R}^2}|\nabla v_q|^2\,dx+\int_{\mathbb{R}^2}V|v_q|^2\,dx-a\int_{\mathbb{R}^2}|v_q|^{q+2}\,dx-\mu_q=0.
\]
By the  Br\'{e}zis-Lieb lemma,
\[
\int_{\mathbb{R}^2}|\nabla u_n-\nabla v_q|^2\,dx+\int_{\mathbb{R}^2}V|u_n-v_q|^2\,dx=a\int_{\mathbb{R}^2}|u_n-v_q|^{q+2}\,dx+o(\|u_n\|_{\mathcal{H}}).
\]
Since $\mathcal{H}\hookrightarrow L^q(\mathbb{R}^2)$ is compact for any $q\geq2$, we have
\begin{equation}\label{qsl}
u_n\rightarrow v_q\ \  {\rm in} \ \ \mathcal{H}.
\end{equation}
Thus, $E_{a,q}(v_q)=c_q$ and $\|v_q\|_2=1$. Furthermore, by Lemma \ref{lem:1.1}, $\lim_{n\to\infty}\|v_n-v_q\|_{\mathcal{H}}=0$. Noting $v_n\geq 0$ and $c_q>0$, we have $v_q\geq 0$ and $v_q\neq0$. By the strong maximum principle, we conclude $v_q>0$.  This ends the proof.
\end{proof}

\bigskip

{\bf Proof of Theorem \ref{thm1}.} The results in Theorem \ref{thm1} follow by Propositions \ref{prop:2.1} and \ref{prop:2.2}. $\Box$

\bigskip

Now, we study the asymptotic behavior of critical points $u_q$ and $v_q$ as $q\to 2_+$.

\begin{Lemma}\label{lem:1.1a}
The minimizers $\{u_q\}$ is uniformly bounded in $\mathcal{H}$ for $q> 2$ close to 2.
\end{Lemma}

\begin{proof} We argue indirectly. Suppose $\{u_q\}$ is not uniformly bounded,  there would exist $\{q_k\}$ with $q_k\to 2_+$ as $k\to \infty$ such that
\begin{equation}\label{3.29a1}
\lim_{k\to\infty}\|u_{q_k}\|_{\mathcal{H}}^2=+\infty.
\end{equation}

Let $\varphi\in C_c^\infty(\mathbb{R}^2)$ be such that
\begin{equation*}
0\leq \varphi\leq 1 \ \ {\rm and} \ \ \|\varphi\|_2^2=1.
\end{equation*}
Since ${\tau_q^2}$ given in \eqref{eq:1.11} tends to infinity as $q\to 2_+$,  we have $\varphi\in A_{{\tau_q^2}}$, for $q>2$ close to 2, and
\begin{equation}\label{eq:2.24}
E_{a,q}(u_q)=\inf_{u\in A_{{\tau_q^2}}}E_{a,q}(u_q)
\leq E_{a,q}(\varphi)
\leq \frac{1}{2}\int_{\mathbb{R}^2}(|\nabla \varphi|^2+V|\varphi|^2)\,dx.
\end{equation}

By  Lemma \ref{lem:5.3},  $Q_q(u_q)=0$, we find from \eqref{eq:2.21} and \eqref{eq:2.24} that
\begin{equation*}
\begin{split}
&\frac{q_k-2}{2}\int_{\mathbb{R}^2}|\nabla u_{q_k}|^2\,dx+\frac{1}{2}\int_{\mathbb{R}^2}Vu_{q_k}^2\,dx+\frac{1}{2}\int_{\mathbb{R}^2}[(q-1)V+x\cdot \nabla V]u_{q_k}^2\,dx\\
&=q_kE_{a,q_k}\leq \frac{3}{2}\int_{\mathbb{R}^2}(|\nabla \varphi|^2+V\varphi^2)\,dx.
\end{split}
\end{equation*}
The assumption $(V2)$ implies
\begin{equation}\label{3.29b}
\int_{\mathbb{R}^2}Vu_{q_k}^2\,dx\leq C,
\end{equation}
and
\begin{equation}\label{3.29c}
\int_{\mathbb{R}^2}|\nabla u_{q_k}|^2\,dx\leq \frac{C}{q_k-2}.
\end{equation}
By \eqref{3.29a1} and \eqref{3.29b}, we have
\begin{equation}\label{3.29d}
\lim_{k\to \infty}\|\nabla u_{q_k}\|_2^2=+\infty.
\end{equation}
Let
\begin{equation}\label{3.29e}
\eta_{q_k}=\|\nabla u_{q_k}\|_2^{-1}.
\end{equation}
and
\begin{equation}\label{3.29f}
\tilde{f}_{q_k}(x)=\eta_{q_k}u_{q_k}(\eta_{q_k}x).
\end{equation}
Then
\begin{equation}\label{3.29g}
\|\nabla \tilde{f}_{q_k}\|_2^2=\| \tilde{f}_{q_k}\|_2^2=1.
\end{equation}
By $(V2)$ and \eqref{3.29b}, we have
\begin{equation}\label{3.29h}
\bigg|\int_{\mathbb{R}^2}x\cdot\nabla Vu_{q_k}^2\,dx\bigg|\leq \int_{\mathbb{R}^2}|x\cdot\nabla V|u_{q_k}^2\,dx\leq \int_{\mathbb{R}^2}C(V+1)u_{q_k}^2\,dx\leq C,
\end{equation}
and  Lemma \ref{lem:5.3} implies
\begin{equation}\label{3.41a}
\int_{\mathbb{R}^2}|\nabla u_{q_k}|^2\,dx-\frac{1}{2}\int_{\mathbb{R}^2}x\cdot\nabla V|u_{q_k}|^2\,dx=\frac{aq_k}{q_k+2}\int_{\mathbb{R}^2}|u_{q_k}|^{q_k+2}\,dx.
\end{equation}
This with \eqref{3.29d} and \eqref{3.29h} yields that
\begin{equation}\label{3.29i}
\lim_{k\to \infty}\frac{\int_{\mathbb{R}^2}|\nabla u_{q_k}|^2\,dx}{a\int_{\mathbb{R}^2}|u_{q_k}|^{q_k+2}\,dx}=\frac{1}{2}.
\end{equation}
By \eqref{3.29d}, for $k$ large enough, $\|\nabla u_{q_k}\|_2^2\geq 1$, we deduce from \eqref{3.29c} and \eqref{3.29e} that
\[
1\geq \eta_{q_k}^{q_k-2}=\|\nabla u_{q_k}\|_2^{2-q_k}\geq \big(\frac{C}{q_k-2}\big)^{\frac{2-q_k}{2}}\to 1\ \ {\rm as} \ \ k\to \infty.
\]
Hence
\begin{equation}\label{3.29j}
\eta_{q_k}^{q_k-2}\to 1\ \ {\rm as} \ \ k\to \infty.
\end{equation}
It follows from \eqref{3.29e}, \eqref{3.29i} and \eqref{3.29j} that there exist $C_1>0$ and $C_2>0$ such that
\begin{equation}\label{3.29jj}
C_1\leq \int_{\mathbb{R}^2}|\tilde{f}_{q_k}|^{{q_k}+2}\,dx=\eta_{q_k}^{{q_k}-2}
\frac{\int_{\mathbb{R}^2}|u_{q_k}|^{{q_k}+2}\,dx}{\int_{\mathbb{R}^2}|\nabla u_{q_k}|^2\,dx}
\leq C_2.
\end{equation}
We claim from \eqref{3.29g} and \eqref{3.29jj} that, there exist $\{y_{q_k}\}\subset\mathbb{R}^2$, $R_0>0$ and $\eta>0$ such that
\begin{equation}\label{3.29kkkk}
\liminf_{k\to\infty}\int_{B_{R_0}(y_{q_k})}|\tilde{f}_{q_k}|^2\,dx\geq \eta.
\end{equation}
Indeed, if it is not the case, for any $R>0$, there would exist a sequence, still denoted by  $\{\tilde{f}_{q_k}\}$ such that
\[
\lim_{k\to\infty} \sup_{y\in \mathbb{R}^2}\int_{B_{R}(y)}|\tilde{f}_{q_k}|^2\,dx=0.
\]
By the vanishing lemma, see for instance Lemma 1.21 in \cite{Wi}, we have $\tilde{f}_{q_k}\to 0$ strongly in $L^\gamma(\mathbb{R}^2)$ for any $\gamma>2$, which contradicts \eqref{3.29jj}.

Denote $f_{q_k}=\tilde{f}_{q_k}(x+y_{q_k})=\eta_{q_k}u_{q_k}(\eta_{q_k}(x+y_{q_k}))$. We have
\begin{equation}\label{3.29jjh}
\begin{split}
&\|\nabla f_{q_k}\|_2^2=\| f_{q_k}\|_2^2=1;\\
&C_1\leq \int_{\mathbb{R}^2}|f_{q_k}|^{{q_k}+2}\,dx\leq C_2;\\
&\liminf_{k\to\infty}\int_{B_{R_0}(0)}|f_{q_k}|^2\,dx\geq \eta.
\end{split}
\end{equation}
Hence, there exist a sequence $\{q_k\}$ and $f\in H^1(\mathbb{R}^2)$ such that $f_{q_k}\rightharpoonup f$ weakly in $H^1(\mathbb{R}^2)$ and $f_{q_k}\to f$ strongly in
$L_{loc}^\gamma(\mathbb{R}^2)$ for any $\gamma\geq 2$. Noting that by  \eqref{3.29jjh}, $f\neq0$. Since $u_{q_k}$ solves \eqref{eq:1.6c} for the Lagrange multiplier $\mu_{q_k}$, we see that $f_{q_k}$ solves
\begin{equation}\label{3.29hhh}
\begin{split}
-\Delta f_{q_k}+\eta_{q_k}^2V(\eta_{q_k}(x+y_{q_k}))f_{q_k}
=\eta_{q_k}^2\mu_{q_k}f_{q_k}+\eta_{q_k}^{2-q}af_{q_k}^{{q_k}+1}.
\end{split}
\end{equation}
Next, we prove that $\{\eta_{q_k}y_{q_k}\}$ is uniformly bounded. If it is not the case, there would exist a subsequence of $\{q_k\}$, still denoted by $\{q_k\}$, such that $\lim_{k\to\infty}|\eta_{q_k}y_{q_k}|=+\infty$, then by \eqref{3.29jjh}, we have
\begin{equation*}
\begin{split}
\int_{\mathbb{R}^2}V(x)u_{q_k}^2\,dx=\int_{\mathbb{R}^2}V(\eta_{q_k}(x+y_{q_k}))f_{q_k}^2(x)\,dx\geq \int_{B_{R_0}(0)}V(\eta_{q_k}(x+y_{q_k}))f_{q_k}^2(x)\,dx\to \infty,
\end{split}
\end{equation*}
which is a contradiction to \eqref{3.29b}.
By \eqref{eq:5.14}, \eqref{3.29d}, \eqref{3.29e}, \eqref{3.29b} and \eqref{3.29i}, we get
\begin{equation}\label{3.29k}
\begin{split}
\lim_{k\to\infty}\eta_{q_k}^2\mu_{q_k}&=\lim_{k\to\infty}\frac{\mu_{q_k}}{\|\nabla u_{q_k}\|_2^2}\\
&=1+\lim_{k\to\infty}\frac{\int_{\mathbb{R}^2}V(x)u_{q_k}^2\,dx}{\|\nabla u_{q_k}\|_2^2}-\lim_{k\to\infty}\frac{a\int_{\mathbb{R}^2}|u_{q_k}|^{q_k+2}\,dx}{\|\nabla u_{q_k}\|_2^2}\\
&=-1.
\end{split}
\end{equation}
By \eqref{3.29d}, \eqref{3.29e}, \eqref{3.29j} and \eqref{3.29k}, letting $k\to\infty$ in \eqref{3.29hhh}, we have that $f\neq0$ solves
\[
-\Delta f+f=af^3.
\]
Using Lemma \ref{lem:5.3} with $V=1$, we have
\[
\int_{\mathbb{R}^2}|\nabla f|^2\,dx=\frac{a}{2}\int_{\mathbb{R}^2}f^4\,dx.
\]
By the Gagliardo-Nirenberg inequality (see \cite{W}) and $\|f\|_2^2\leq  1$, we obtain
\begin{equation*}
\begin{split}
\int_{\mathbb{R}^2}f^4\,dx\leq \frac{2}{a^*}\int_{\mathbb{R}^2}|\nabla f|^2\,dx\int_{\mathbb{R}^2}f^2\,dx=\frac{a}{a^*}\int_{\mathbb{R}^2}f^4\,dx\int_{\mathbb{R}^2}f^2\,dx\leq\frac{a}{a^*}\int_{\mathbb{R}^2}f^4\,dx,
\end{split}
\end{equation*}
which is a contradiction since $a<a^*$ and $f\neq0$. The assertion follows.

\end{proof}

Now, we are in position to prove Theorem \ref{thm2}.
\bigskip

{\bf Proof of Theorem \ref{thm2}.}  We first prove $(i)$.

By Lemma \ref{lem:1.1a}, there exists $u_0\in \mathcal{H}$ such that $u_q\rightharpoonup u_0$ weakly in $\mathcal{H}$, and Lemma \ref{lem:2.1} implies that $u_q\rightarrow u_0$ in $L^p(\mathbb{R}^2)$ for any $p\geq 2$.
Now we prove $u_q\to u_0$ in $\mathcal{H}$ as $q\to 2$.
Since $u_q$ satisfies
\begin{equation}\label{eq:2.25}
-\Delta u_q+V(x)u_q=\mu_qu_q+a(u_q)^{q+1},
\end{equation}
where
\begin{equation}\label{eq:2.25a}
\mu_q=\int_{\mathbb{R}^2}\big(|\nabla u_q|^2+V(x)|u_q|^2\big)\,dx-a\int_{\mathbb{R}^2}|u_q|^{q+2}\,dx,
\end{equation}
we see that $\mu_q$ is uniformly bounded in $q$.  Suppose  $\mu_q\rightarrow \mu_0$ as $q\to 2_+$, then $u_0$ satisfies
\begin{equation}\label{eq:2.26}
-\Delta u_0+V(x)u_0=\mu_0 u_0+au_0^{3}
\end{equation}
implying
\[
\int_{\mathbb{R}^2}|\nabla u_0|^2+V(x)|u_0|^2 \,dx=\mu_0 +a\int_{\mathbb{R}^2}|u_0|^4\,dx.
\]
We deduce from the Br\'{e}zis-Lieb lemma and
\[
\int_{\mathbb{R}^2}|\nabla u_q|^2+V(x)|u_q|^2 \,dx=\mu_q+a\int_{\mathbb{R}^2}|u_q|^{q+2}\,dx
\]
that
\begin{equation*}
\begin{split}
&\int_{\mathbb{R}^2}|\nabla u_q-\nabla u_0|^2\,dx+\int_{\mathbb{R}^2}V(x)|u_q-u_0|^2\,dx\\
&=\mu_q-\mu_0+a\int_{\mathbb{R}^2}|u_q-u_0|^{q+2}\,dx+a\int_{\mathbb{R}^2}|u_0|^{q+2}\,dx
-a\int_{\mathbb{R}^2}|u_0|^4\,dx+o(1),
\end{split}
\end{equation*}
as  $q\to 2_+$. This implies  $u_q\rightarrow u_0$ in $\mathcal{H}$.

We claim that  $u_0$  is a minimizer of
\[
d_a(2)=\inf_{u\in \mathcal{H},\|u\|_2 =1}E_{a,2}(u).
\]
In fact, suppose on the contrary that
$$
E_{a,2}(u_0)>d_a(2)=E_{a,2}(w),
$$
where $w$ is a minimizer of $d_a(2)$. By the convergence of $u_q\to u_0$ in $\mathcal{H}$,  for $q$ close to $2$ there exists  $\varepsilon_0\in (0,\frac{1}{3}\big(E_{a,2}(u_0)-E_{a,2}(w)\big))$   such that
\[
E_{a,2}(w)\geq E_{a,q}(w)-\varepsilon_0,\quad E_{a,2}(u_0)\leq E_{a,q}(u_q)+\varepsilon_0
\]
and
\begin{equation}\label{eq:2.27}
\|\nabla u_0\|_2^2+\|\nabla w\|_2^2\leq {\tau_q^2}.
\end{equation}
As a result,
  \begin{equation}\label{eq:2.28}
  E_{a,q}(u_q)\geq E_{a,2}(u_0)-\varepsilon_0>E_{a,2}(w)+\varepsilon_0\geq E_{a,q}(w).
  \end{equation}
It yields
  \[
  E_{a,q}(u_q)>E_{a,q}(w)\geq \inf_{u\in A_{{\tau_q^2}}}E_{a,q}(u),
  \]
which is a contradiction to the fact $E_{a,q}(u_q)= \inf_{u\in A_{{\tau_q^2}}}E_{a,q}(u)$.

\bigskip

Now we prove $(ii)$, that is, $\{v_q\}$ is unbounded in $H^1(\mathbb{R}^2)$.  Suppose by the contradiction that $\{v_q\}$ is bounded in $H^1(\mathbb{R}^2)$, then by the Sobolev embedding theorem, $\{v_q\}$ is bounded in $L^{q+2}(\mathbb{R}^2)$.
Noting equation \eqref{eq:2.9} implies that
 $$
 c_q=E_{a,q}(v_q)\geq\inf_{u\in \partial A_{{\tau_q^2}}}E_{a,q}|_{V=0}(u)\geq \frac{q-2}{2q}{\tau_q^2}\rightarrow+\infty
 $$
 as $q\to 2_+$, we infer that
\[
\int_{\mathbb{R}^2}V(x)|v_q|^2\,dx\rightarrow +\infty
\]
as $ q\to 2_+$. This and \eqref{eq:2.25a} yield $\mu_q\rightarrow\infty$, as $q\to 2_+$. However, this is impossible. In fact,  since $-\Delta +V$ is a compact operator,
so it has a discrete spectrum,  and the first eigenvalue $\lambda_1$  and corresponding eigenfunction $\varphi_1$ are positive, which satisfy
\[
-\Delta \varphi_1+V\varphi_1=\lambda_1 \varphi_1.
\]
Hence, we obtain
  \begin{equation}\label{eq:2.29}
  \begin{split}
  \int_{\mathbb{R}^2}\mu_qv_q\varphi_1\,dx&=\int_{\mathbb{R}^2}(-\Delta v_q+Vv_q-av_q^{q+1})\varphi_1\,dx\\
  &< \int_{\mathbb{R}^2}(-\Delta \varphi_1+V\varphi_1)v_q\,dx\\
  &=\int_{\mathbb{R}^2}\lambda_1\varphi_1 v_q\,dx.
  \end{split}
  \end{equation}
That is
\begin{equation}\label{lg1}
\mu_q< \lambda_1,
\end{equation}
which is a contradiction. $\Box$

\bigskip

Finally, we show that the assumption $(V_e)$  implies $(V_1)$ and $(V_2)$, that is, the result of Corollary \ref{cor1}. Therefore, the consequences in Theorem \ref{thm1} hold true if we assume $(V_e)$.

\bigskip

{\bf Proof of Corollary \ref{cor1}} We will verify that $V$ satisfies condition $(V_1)$ and $(V_2)$. Obviously, $(V_1)$ holds true. So we need only to verify $(V_2)$.
It holds that
\[
x\cdot \nabla V=\sum_{i=1}^np_iV+\sum_{i=1}^np_ix_i\prod_{m\neq i}|x-x_m|^{p_m}|x-x_i|^{p_i-2}(x-x_i).
\]
Let
\begin{equation}\label{eq:3.3a}
f_i=p_ix_i\prod_{m\neq i}|x-x_m|^{p_m}|x-x_i|^{p_i-2}(x-x_i).
\end{equation}
Choose $\delta_i>0$ such that
$$
\sum_{i=1}^n|p_ix_i\delta_i|<\frac p2.
$$
If $|x-x_i|\geq \frac{1}{\delta_i}$, we have $|f_i|\leq |p_ix_i\delta_i|V$; and if $|x-x_i|\leq \frac{1}{\delta_i}$, then $|f_i|\leq C_i$ for some $C_i>0$. Hence,
there exists $C>0$ such that
\[
|x\cdot \nabla V(x)|\leq C (V(x)+1).
\]
and
\begin{equation}\label{eq:3.3aa}
x\cdot \nabla V\geq \sum_{i=1}^np_iV-\sum_{i=1}^n(|p_ix_i\delta_i|V+C_i)\geq \frac p2 V-\sum_{i=1}^n C_i.
\end{equation}
Thus,  $(V_2)$ immediately follows. $\Box$

\bigskip

\section{Energy estimates}

\bigskip

It is known from Theorem \ref{thm2} that the local minimizer $u_q$ of $E_{a,q}$ tends to a minimizer $u_0\in H^1(\mathbb{R}^2)$ of $d_a(2)$.
Let $v_q$ be the mountain pass point of $E_{a,q}$ obtained by Corollary \ref{cor1}.  In this section,  we focus on investigating the asymptotic behavior of energy $E_{a,q}(v_q)$ as $q\to2_+$.
We suppose in this and next sections that $V$ satisfies condition $(V_e)$.

\begin{Proposition}\label{lem:3.2} There holds
\[
E_{a,q}(v_q)=\frac{q-2}{2q}\Big(\frac{2a_q^*}{qa}\Big)^{\frac{2}{q-2}}+o(1)
\]
as $q\rightarrow 2_+$.
\end{Proposition}

\bigskip

\begin{proof}
It follows from \eqref{eq:2.9} and the definition of $c_q$ that
 \begin{equation*}
c_q=E_{a,q}(v_q)\geq \inf_{u\in \partial A_{\tau_q^2}} \tilde E_{a,q}(v_q)\geq \frac{q-2}{2q}\Big(\frac{2a_q^*}{qa}\Big)^{\frac{2}{q-2}}.
\end{equation*}

 Now, we prove
  \begin{equation*}
 E_{a,q}(v_q)\leq \frac{q-2}{2q}\Big(\frac{2a_q^*}{qa}\Big)^{\frac{2}{q-2}}+o(1)
 \end{equation*}
 as  $q\to 2_+.$ To this purpose, we will construct a path $g$ in $\Gamma_q$ defined in \eqref{eq:2.17} connecting $\varphi^{t_0}$ and $\varphi^{t_1}$ so that
 \[
 E_{a,q}(v_q)\leq \max_{t\geq 0}E_{a,q}(g(t))\leq \frac{q-2}{2q}\Big(\frac{2a_q^*}{qa}\Big)^{\frac{2}{q-2}}+o(1).
 \]

The path $g$ is constructed in three parts. First,  we construct a path $g_1$ connecting $\varphi^{t_0}$ to some $w_q^{\tilde{t}_0}$, and estimate $E_{a,q}(g_1(s))$.

 Let $\lambda_i\in(0,\infty]$ be given by
\[
\lambda_i=\lim_{x\rightarrow x_i}\frac{V(x)}{|x-x_i|^{p_{max}}}.
\]
 Define $\lambda=\min\{\lambda_1,\cdot\cdot\cdot, \lambda_n\}$, and denote $\mathcal{Z}:=\{x_i|\lambda_i=\lambda\}$ and  $p = \sum_{i=1}^n p_i$.
  Let
  \begin{equation}\label{4.1a}
  w_q(x)=\frac{\tau_q\varphi_q(\tau_q(x-x_0))}{\|\varphi_q\|_2},
  \end{equation}
  where $\varphi_q$ is the unique positive solution of \eqref{eq:1.9} and $x_0\in \mathcal{Z}$. Then,
\begin{equation}\label{4.1b}
\int_{\mathbb{R}^2}|w_q|^2\,dx=1.
\end{equation}
By \eqref{eq:3.3} and \eqref{eq:1.6a}, we deduce that
\begin{equation}\label{eq:3.13}
\int_{\mathbb{R}^2}|\nabla w_q|^2\,dx=\tau_q^2\frac{\int_{\mathbb{R}^2}|\nabla \varphi_q|^2\,dx}{\int_{\mathbb{R}^2}|\varphi_q|^2\,dx}=\tau_q^2,
\end{equation}
and
\begin{equation}\label{eq:3.14}
\int_{\mathbb{R}^2}|w_q|^{q+2}\,dx=\tau_q^q\frac{\int_{\mathbb{R}^2}|\varphi_q|^{q+2}\,dx}{\|\varphi_q\|_2^{q+2}}=\frac{q+2}{2a_q^*}\tau_q^q.
\end{equation}
Let
\begin{equation}\label{4.2a}
\tilde t_0^2=\frac{1}{6}\theta:=\frac{1}{6}\frac{q-2}{2q}
\end{equation}
and
\begin{equation}\label{4.2b}
w_q^{\tilde{t}_0}(x)=\tilde{t}_0w_q(\tilde{t}_0x).
\end{equation}
By \eqref{eq:3.13} we have
\begin{equation}\label{eq:3.15}
\int_{\mathbb{R}^2}|\nabla w_q^{\tilde{t}_0}|^2\,dx=\tilde t_0^2\int_{\mathbb{R}^2}|\nabla w_q|^2\,dx=\frac{1}{6}\theta \tau_q^2.
\end{equation}
Define a path $g_1$ connecting $(w_q)^{\tilde{t_0}}$ and $\varphi^{t_0}$ as follows.
\begin{equation*}
g_1(s):=\frac{s w_q^{\tilde{t}_0}+(1-s)\varphi^{t_0}}{\|s w_q^{\tilde{t}_0}+(1-s)\varphi^{t_0}\|_2}, \ \, s\in [0,1].
\end{equation*}
We may verify that
\begin{equation}\label{eq:3.16}
\begin{split}
\|s w_q^{\tilde{t}_0}+(1-s)\varphi^{t_0}\|_2^2&\geq s^2\| w_q^{\tilde{t}_0}\|_2^2+(1-s)^2\|\varphi^{t_0}\|_2^2\\
&=s^2+(1-s)^2\geq \frac{1}{2}.
\end{split}
\end{equation}
By \eqref{eq:2.11} and \eqref{eq:3.15}, we have
\begin{equation}\label{eq:3.17}
\begin{split}
&\|s\nabla  w_q^{\tilde{t}_0}+(1-s)\nabla \varphi^{t_0}\|_2^2\\
\leq &2(s^2\|\nabla  w_q^{\tilde{t}_0}\|_2^2+(1-s)^2\|\nabla \varphi^{t_0}\|_2^2)\\
\leq & 2(\|\nabla  w_q^{\tilde{t}_0}\|_2^2+\|\nabla \varphi^{t_0}\|_2^2)\\
= & 2\big(\frac{1}{6}\theta \tau_q^2+\frac{1}{4}\theta \tau_q^2\big)=\frac{5}{6}\theta\tau_q^2.
\end{split}
\end{equation}
It follows from \eqref{eq:3.16} and \eqref{eq:3.17} that
\begin{equation}\label{eq:3.18}
E_{a,q}|_{V=0}(g_1(s))\leq \frac{1}{2}\|\nabla g_1(s)\|_2^2\leq \frac{5}{6}\theta\tau_q^2.
\end{equation}

On the other hand, by \eqref{eq:3.16},
\begin{equation}\label{eq:3.19}
\begin{split}
\int_{\mathbb{R}^2}V(x)|g_1(s)|^2\,dx&\leq 2\int_{\mathbb{R}^2}V(x)|sw_q^{\tilde{t}_0}+(1-s)\varphi^{t_0}|^2\,dx\\
&\leq 4\Big(\int_{\mathbb{R}^2}V(x)|w_q^{\tilde{t}_0}|^2\,dx+\int_{\mathbb{R}^2}V(x)|\varphi^{t_0}|^2\,dx\Big).
\end{split}
\end{equation}
Recall
\[
V\bigg(\frac{x}{t}\bigg)=\bigg|\frac{x}{t}-x_1\bigg|^{p_1}\bigg|\frac{x}{t}-x_2\bigg|^{p_2}\cdot\cdot\cdot\bigg|\frac{x}{t}-x_n\bigg|^{p_n},
\]
we claim that
\begin{equation}\label{eq:3.20}
V\bigg(\frac{x}{t}\bigg)\leq Ct^{-\sum_{i=1}^n p_i}|x|^{\sum_{i=1}^n p_i}+C = C t^{-p}|x|^p +C,
\end{equation}
where $C>0$ is independent of $q$.
Indeed, if $\big|\frac{x}{t}\big|\geq \max_i|x_i|$, then
$$
\big|\frac{x}{t}-x_i\big|\leq \big|\frac{x}{t}\big|+|x_i|\leq 2\big|\frac{x}{t}\big|
$$
implies
$$
V\bigg(\frac{x}{t}\bigg)\leq 2^p\big|\frac{x}{t}\big|^{p}.
$$
If $\big|\frac{x}{t}\big|\leq \max_i|x_i|$, we have
\[
\big|\frac{x}{t}-x_i\big|\leq \big|\frac{x}{t}\big|+|x_i|\leq 2\max_i|x_i|
\]
 and then
 \[
 V\bigg(\frac{x}{t}\bigg)\leq \big(2\max|x_i|\big)^{p}.
 \]
Consequently,
\[
V\big(\frac{x}{t}\big)\leq 2^{p}\big|\frac{x}{t}\big|^{p}+\big(2\max|x_i|\big)^{p},
\]
that is, the claim is valid.
By \eqref{4.1b} and \eqref{eq:3.20}, we deduce
\begin{equation}\label{eq:3.21}
\begin{split}
&\int_{\mathbb{R}^2}V(x)\big|w_q^{\tilde{t}_0}\big|^2\,dx\\
&=\int_{\mathbb{R}^2}V\big(\frac{x}{\tilde{t}_0}\big)|w_q|^2\,dx\\
&\leq C\int_{\mathbb{R}^2}\tilde{t}_0^{-p}|x|^{p}|w_q|^2\,dx+\tilde{C}\\
&=\frac{\tilde{t_0}^{-p}}{\|\varphi_q\|_2^2}\int_{\mathbb{R}^2}|\frac{x}{\tau_q}+x_0|^p\varphi_q^2(x)\,dx+\tilde{C}.
\end{split}
\end{equation}

By \eqref{eq:5.1}  and the  Lebesgue dominated theorem, we get
\begin{equation}\label{eq:3.22}
\int_{\mathbb{R}^2}\big|\frac{x}{\tau_q}+x_0\big|^{p}\frac{|\varphi_q(x)|^2}{\|\varphi_q\|_2^2}\,dx\rightarrow |x_0|^p
\end{equation}
since $\tau_q\to\infty$ as $q\to 2_+$.
It follows from \eqref{eq:3.21} and \eqref{eq:3.22} that for $q\to 2_+$,
\begin{equation}\label{eq:3.23}
\int_{\mathbb{R}^2}V(x)|w_q^{\tilde{t}_0}|^2\,dx\leq O(\tilde{t}_0^{-p}).
\end{equation}

By \eqref{eq:2.12}, \eqref{eq:3.19} and  \eqref{eq:3.23}, we have
\begin{equation}\label{eq:3.24}
\int_{\mathbb{R}^2}V(x)|g_1(s)|^2\,dx\leq O(\tilde{t}_0^{-p})+8V(0).
\end{equation}
Taking into account  \eqref{4.2a}, \eqref{eq:3.18} and \eqref{eq:3.24}, we obtain
\begin{equation}\label{eq:3.25}
\begin{split}
E_{a,q}(g_1(s))&\leq \frac{5}{6}\theta \tau_q^2+O(\tilde{t}_0^{-p})+8V(0)\\
&=\frac{5}{6}\frac{q-2}{2q}\Big(\frac{2a_q^*}{qa}\Big)^{\frac{2}{q-2}}+O((q-2)^{-\frac{p}{2}})+8V(0).\\
\end{split}
\end{equation}

Next, we construct the second part of the path $g$.
Let
\begin{equation}\label{4.13a}
\tilde{t}_1=\tau_q^{-1}t_1
\end{equation}
and
\begin{equation}\label{4.13b}
w_q^{\tilde{t}_1}=\tilde{t}_1w_q(\tilde{t}_1x),
\end{equation}
where $t_1$ is defined in \eqref{t1} and $w_q$ is defined in \eqref{4.1a}. We define a path $g_2$ connecting $w_q^{\tilde{t}_1}$ and $\varphi^{t_1}$ as follows.
\begin{equation*}
g_2(s)=\frac{sw_q^{\tilde{t}_1}+(1-s)\varphi^{t_1}}{\|sw_q^{\tilde{t}_1}+(1-s)\varphi^{t_1}\|_2},\ \ s\in [0,1].
\end{equation*}

By \eqref{eq:3.13} and \eqref{eq:3.16}, we get
\begin{equation*}
\begin{split}
\|\nabla g_2(s)\|_2^2&\leq 2\|s\nabla w_q^{\tilde{t}_1}+(1-s)\nabla \varphi^{t_1}\|_2 \\
&\leq 4(s^2\|\nabla w_q^{\tilde{t}_1} \|_2^2+(1-s)^2\|\nabla\varphi^{t_1}\|_2^2\\
&=4(s^2\tilde{t}_1^2\tau_q^2+(1-s)^2t_1^2\|\nabla \varphi\|_2^2).
\end{split}
\end{equation*}
Since $\|w_q^{\tilde{t}_1}\|_2^2=\|\varphi^{t_1}\|_2^2=1$, we have
\begin{equation*}
\|sw_q^{\tilde{t}_1}+(1-s)\varphi^{t_1}\|_2^2\leq 2[\|sw_q^{\tilde{t}_1}\|_2^2+\|(1-s)\varphi^{t_1}\|_2^2]\leq  4.
\end{equation*}
This   with \eqref{eq:3.14} yields that
\begin{equation}\label{eq:3.26}
\begin{split}
\int_{\mathbb{R}^2}|g_2(s)|^{q+2}\,dx&\geq 2^{-(q+2)}\int_{\mathbb{R}^2}|sw_q^{\tilde{t}_1}+(1-s)\varphi^{t_1}|^{q+2}\,dx\\
&\geq 2^{-q-3}\Big(s^{q+2}\int_{\mathbb{R}^2}|w_q^{\tilde{t}_1}|^{q+2}\,dx+(1-s)^{q+2}\int_{\mathbb{R}^2}|\varphi^{t_1}|^{q+2}\,dx\Big)\\
&= 2^{-q-3}\bigg(\frac{q+2}{2a_q^*}s^{q+2}\tilde{t}_1^q\tau_q^q+(1-s)^{q+2}t_1^q\int_{\mathbb{R}^2}|\varphi|^{q+2}\,dx\bigg).
\end{split}
\end{equation}
Therefore,
\begin{equation*}
\begin{split}
E_{a,q}|_{V=0}(g_2(s))&\leq 2(s^2\tilde{t}_1^2\tau_q^2+(1-s)^2t_1^2\|\nabla \varphi\|_2^2)\\
&-\frac{a}{2^{q+3}(q+2)}\bigg(\frac{q+2}{2a_q^*}s^{q+2}\tilde{t}_1^q\tau_q^q+(1-s)^{q+2}t_1^q\int_{\mathbb{R}^2}|\varphi|^{q+2}\,dx\bigg).\\
\end{split}
\end{equation*}
If $0<s\leq \frac{1}{2}$,  we have
\begin{equation*}
\begin{split}
E_{a,q}|_{V=0}(g_2(s))&\leq 2\tilde{t}_1^2\tau_q^2+2t_1^2\|\nabla \varphi\|_2^2-\frac{a}{2^{q+3}(q+2)}(1-s)^{q+2}t_1^q\int_{\mathbb{R}^2}|\varphi|^{q+2}\,dx\\
&\leq2t_1^2+2t_1^2\|\nabla \varphi\|_2^2-\frac{a}{2^{2q+5}(q+2)}t_1^q\int_{\mathbb{R}^2}|\varphi|^{q+2}\,dx.\\
\end{split}
\end{equation*}
Since $t_1^q/t_1^2=2^{\frac{1}{q-2}}C^{q-2}\to \infty$ as $q\to 2+$, by \eqref{t1} we find for
$q$ close to $2$ that
\[
E_{a,q}|_{V=0}(g_2(s))
\leq -\frac{aC}{2^{2q+6}q+2}2^{\frac{q}{(q-2)^2}}\int_{\mathbb{R}^2}|\varphi|^{q+2}\,dx.
\]
If $\frac{1}{2}\leq s\leq 1$, we deduce  in the same way that for $q$ close to $2$
\begin{equation*}
\begin{split}
E_{a,q}|_{V=0}(g_2(s))&\leq 2(s^2\tilde{t}_1^2\tau_q^2+(1-s)^2t_1^2\|\nabla \varphi\|_2^2)-\frac{a}{2^{q+3}(q+2)}\frac{q+2}{2a_q^*}s^{q+2}\tilde{t}_1^q\tau_q^q\\
&\leq 2(t_1^2+t_1^2\|\nabla \varphi\|_2^2)-\frac{a}{2^{2q+5}(q+2)}\frac{q+2}{2a_q^*}t_1^q\\
&\leq-\frac{a}{2^{2q+6}(q+2)}\frac{q+2}{2a_q^*}t_1^q.
\end{split}
\end{equation*}
Thus,
\begin{equation}\label{eq:3.27}
E_{a,q}|_{V=0}(g_2(s))\leq -Ct_1^q=-2^{\frac{q-1}{(q-2)^2}}C.
\end{equation}
We may prove as \eqref{eq:3.23} that
\begin{equation}\label{eq:3.28}
\int_{\mathbb{R}^2}V(x)|g_2(s)|^2\,dx\leq C.
\end{equation}
Consequently,
\begin{equation}\label{eq:3.29}
E_{a,q}(g_2(s))<0
\end{equation}
for $q>2$ and close to $2$.

Finally, we construct a path linking $w_q^{\tilde{t}_0}$ and $w_q^{\tilde{t}_1}$.

Since \eqref{eq:3.15} and $\|\nabla w_q^{\tilde{t}_1}\|_2^2=2^{\frac{2}{(q-2)^2}}C$, there exists a $\tilde{t}_2\in (\tilde{t}_0, \tilde{t}_1)$ such that  $\|\nabla w_q^{\tilde{t}_2}\|=\tau_q^2$. By \eqref{eq:2.9}, we have
\begin{equation}\label{eq:3.30}
E_{a,q}(w_q^{\tilde{t}_2})\geq E_{a,q}|_{V=0}(w_q^{\tilde{t}_2})\geq \theta \tau_q^2.
\end{equation}
Define a path linking $w_q^{\tilde{t}_0}$ and $w_q^{\tilde{t}_1}$ as follows.
\[
g_3(s)=w_q^{s\tilde{t}_0+(1-s)\tilde{t}_1},\ \ s\in [0,1].
\]

Now, we define a path $g$ in $\Gamma_q$ defined in \eqref{eq:2.16} by $g_1,\,g_2$ and $g_3$. Precisely, we define $g(s)=g_1(3s)$ if $0\leq s\leq \frac 13$; $g(s) = g_3(3s-1)$ if $\frac 13\leq s\leq \frac 23$, and $g(s)=g_2(3s-2)$ if $\frac 23\leq s\leq 1$.
Then, $g(s)\in \Gamma_q$.

By \eqref{eq:3.25}, \eqref{eq:3.29} and \eqref{eq:3.30}, we have
\begin{equation}\label{eq:3.31}
\max_{\{0\leq s\leq 1\}} E_{a,q}(g(s))=\max_{\{\frac 13\leq s\leq \frac 23\}} E_{a,q}(g_3(3s-1))=\max_{\{\tilde{t}_0\leq t\leq \tilde{t}_1\}}E_{a,q}(w_q^t).
\end{equation}
Apparently, by \eqref{eq:3.13} and \eqref{eq:3.14}, there exists $t_{1,q}\in (\tilde{t}_0, \tilde{t}_1)$ such that
\begin{equation}\label{eq:3.32}
E_{a,q}(w_q^{t_{1,q}})=\max_{\{\tilde{t}_0\leq t\leq \tilde{t}_1\}}E_{a,q}(w_q^t)=:\max_{\{\tilde{t}_0\leq t\leq \tilde{t}_1\}} f(t)
\end{equation}
with $f'(t_{1,q})=0$, where
\[
f(t)=\frac{1}{2}t^2\tau_q^2-\frac{1}{q}t^q\tau_q^2+\frac{1}{2}\int_{\mathbb{R}^2}V\bigg(\frac{x}{t}\bigg)|w_q|^2\,dx.
\]
and
\[
f'(t)=t\tau_q^2-t^{q-1}\tau_q^2-\frac{1}{2t^2}\int_{\mathbb{R}^2}x\cdot\nabla V\bigg(\frac{x}{t}\bigg)|w_q|^2\,dx.
\]
Similar to the proof of \eqref{eq:3.20}, we have
\[
\Big|\nabla V\bigg(\frac{x}{t}\bigg)\Big|=\sum_{1}^np_i\Big|\frac{x}{t}-x_i\Big|^{p_i-1}\Pi_{m\neq i}\Big|\frac{x}{t}-x_m\Big|^{p_m}\leq Ct^{-p+1}|x|^{p-1}+C.
\]
Therefore,
\begin{equation*}
\begin{split}
&\Big|\int_{\mathbb{R}^2}x\cdot \nabla V\bigg(\frac{x}{t}\bigg)|w_q|^2\,dx\Big|\\
&\leq \int_{\mathbb{R}^2}Ct^{-p+1}|x|^p|w_q|^2\,dx
+\int_{\mathbb{R}^2}C|x||w_q|^2\,dx\\
&=\int_{\mathbb{R}^2}Ct^{-p+1}|x|^{p}\tau_q^2|\varphi_q(\tau_q(x-x_0))|^2\,dx+\int_{\mathbb{R}^2}C|x|\tau_q^2|\varphi_q(\tau_q(x-x_0))|^2\,dx.
\end{split}
\end{equation*}
By the change of variables, we estimate
\begin{equation*}
\begin{split}
&\int_{\mathbb{R}^2}t^{-p+1}|x|^{p}\tau_q^2|\varphi_q(\tau_q(x-x_0))|^2\,dx\\
&\leq C\int_{\mathbb{R}^2}t^{-p+1}|x|^{p}\tau_q^{-p}|\varphi_q(x)|^2\,dx+Ct^{-p+1}|x_0|^p\int_{\mathbb{R}^2}|\varphi_q|^2\,dx.\\
\end{split}
\end{equation*}
We remark that the integral $\int_{\mathbb{R}^2}|x|^{p}|\varphi_q(x)|^2\,dx$ is finite because $\varphi_q$ exponentially decays at infinity uniformly in $q$. Therefore,
\[
\int_{\mathbb{R}^2}t^{-p+1}|x|^{p}\tau_q^2|\varphi_q(\tau_q(x-x_0))|^2\,dx \leq Ct^{-p+1}\tau_q^{-p}+Ct^{-p+1}\leq Ct^{-p+1}.
\]
Similarly, we have
\[
\int_{\mathbb{R}^2}|x|\tau_q^2|\varphi_q(\tau_q(x-x_0))|^2\,dx \leq C(\tau_q^{-1}+|x_0|)\leq C.
\]
Hence, if $t\geq \tilde{t}_0=\frac{q-2}{12q}$, we have
\begin{equation}\label{eq:3.33}
\frac 1{2t^{2}}\Big|\int_{\mathbb{R}^2}x\cdot \nabla V\bigg(\frac{x}{t}\bigg)|w_q|^2\,dx\Big|\leq Ct^{-p-1}+Ct^{-2}\leq
 C(\frac{q-2}{12q})^{\min\{-p-1,-2\}}.
\end{equation}
This allows us to infer from $f'(t_{1,q})=0$ that
\begin{equation}\label{eq:3.34}
|t_{1,q}\tau_q^2-t_{1,q}^{q-1}\tau_q^2|\leq C(\frac{q-2}{12q})^{\min\{-p-1,-2\}}.
\end{equation}

We claim that  $t_{1,q}\rightarrow 1$ as $q\to2_+$. Indeed, were it not the case, there would exist $\varepsilon_0>0$ small and $q_n\to2_+$ such that either $t_{1,q_n}\geq1+\varepsilon_0$ or $t_{1,q_n}\leq 1-\varepsilon_0$.

If $t_{1,q_n}\geq1+\varepsilon_0$, noting
\[
\lim_{t\to 0_+}\frac{1-(1 +\varepsilon_0)^t}t=-\ln(1 +\varepsilon_0),
\]
we obtain
\begin{equation*}
\begin{split}
&t_{1,q_n}\tau_{q_n}^2-t_{1,q_n}^{q_n-1}\tau_{q_n}^2\\
\leq& t_{1,q_n}\tau_{q_n}^2\big[1-(1+\varepsilon_0)^{q_n-2}\big]\\
\leq& -\frac 12t_{1,q_n}\tau_{q_n}^2(q_n-2)\ln(1+\varepsilon_0)\\
\leq& -\frac 12(1+\varepsilon_0)\ln(1+\varepsilon_0)(q_n-2)\tau_{q_n}^2.
\end{split}
\end{equation*}
This contradicts \eqref{eq:3.34}
since
\begin{equation*}
\lim_{n\to\infty}\frac{(q_n-2)\tau_{q_n}^2}{(\frac{q_n-2}{12q_n})^{\min\{-p-1,-2\}}}=+\infty.
\end{equation*}

Observe $\tilde t_1> t_{1,q}> \tilde t_0 = \frac 16\frac {q-2}{2q}$, if $t_{1,q_n}\leq1-\varepsilon_0$, using the fact that
\[
\lim_{t\to 0_+}\frac{1-(1 -\varepsilon_0)^t}t=-\ln(1 -\varepsilon_0),
\]
we find
\begin{equation*}
\begin{split}
&t_{1,q_n}\tau_{q_n}^2-t_{1,q_n}^{q_n-1}\tau_{q_n}^2\\
&\geq t_{1,q_n}\tau_{q_n}^2\big[1-(1-\varepsilon_0)^{q_n-2}\big]\\
&\geq -\frac 12t_{1,q_n}\tau_{q_n}^2(q_n-2)\ln(1-\varepsilon_0)\\
&\geq -\frac{1}{12}\frac{q_n-2}{2q_n}\tau_{q_n}^2(q_n-2)\ln(1-\varepsilon_0)\\
&=-\frac{(q_n-2)^2}{24q_n}\ln(1-\varepsilon_0)\tau_{q_n}^2\to +\infty,
\end{split}
\end{equation*}
which contradicts \eqref{eq:3.34}. Consequently, $t_{1,q}\rightarrow1$ if $q\to 2$.

Next, we prove further that
\begin{equation}\label{eq:3.37}
|1-t_{1,q}|\leq \tau_q^{-\frac{3}{2}}.
\end{equation}

Since $t_{1,q}\rightarrow1$ as $q\to 2$, we show as \eqref{eq:3.33} that
\[
\frac 1{2t_{1,q}^{2}}\Big|\int_{\mathbb{R}^2}x\cdot \nabla V\bigg(\frac{x}{t_{1,q}}\bigg)|w_q|^2\,dx\Big|\leq Ct_{1,q}^{-p-1}+Ct_{1,q}^{-2}\leq C.
\]
This together with  $f'(t_{1,q})=0$ yields
\[
|t_{1,q}\tau_q^2-t_{1,q}^{q-1}\tau_q^2|\leq C,
\]
that is,
\begin{equation}\label{eq:3.35}
|(1-t_{1,q}^{q-2})\tau_q^2|\leq C.
\end{equation}
By \eqref{eq:2.11}, \eqref{eq:3.13} and \eqref{eq:3.14}, we have
\begin{equation}\label{eq:3.36}
E_{a,q}|_{V=0}(w_1^{t_{1,q}})
=\frac{1}{2}t_{1,q}^2\tau_q^2-\frac{1}{q}t_{1,q}^q\tau_q^2\leq \frac{q-2}{2q}\tau_{q}^2.
\end{equation}

Now we are ready to prove \eqref{eq:3.37}. Suppose on the contrary that \eqref{eq:3.37} does not hold, there would exist $t_{1,q_n}\to 2_+$ such that
either $t_{1,q_n}\leq 1-\tau_{q_n}^{-\frac{3}{2}}$ or $t_{1,q_n}\geq 1+\tau_{q_n}^{-\frac{3}{2}}$.

If $t_{1,q_n}\leq 1-\tau_{q_n}^{-\frac{3}{2}}$, then
\begin{equation}\label{eq:3.38}
|(1-t_{1,q_n}^{q_n-2})\tau_{q_n}^2|=(1-t_{1,q_n}^{q_n-2})\tau_{q_n}^2\geq \big(1-(1-\tau_{q_n}^{-\frac{3}{2}})^{q_n-2}\big)\tau_{q_n}^2.
\end{equation}
The limit
\[
\lim_{n\to\infty}\frac{(q_n-2)\ln(1-\tau_{q_n}^{-\frac{3}{2}})}{\ln(1-\frac{1}{2}(q_n-2)\tau_{q_n}^{-\frac{3}{2}})} =2
\]
implies
\[
(q_n-2)\ln(1-\tau_{q_n}^{-\frac{3}{2}})\leq \ln(1-\frac{1}{2}(q_n-2)\tau_{q_n}^{-\frac{3}{2}})
\]
for $n$ large, namely,
\[
(1-\tau_{q_n}^{-\frac{3}{2}})^{q_n-2}\leq 1-\frac{1}{2}(q_n-2)\tau_{q_n}^{-\frac{3}{2}}.
\]
Hence, inequality \eqref{eq:3.38} yields that
\begin{equation*}
|(1-t_{1,q_n}^{q_n-2})\tau_{q_n}^2|\geq [1-(1-2^{-1}(q_n-2)\tau_{q_n}^{-\frac{3}{2}})]\tau_{q_n}^2=2^{-1}(q_n-2)\tau_{q_n}^{\frac{1}{2}}\rightarrow \infty,
\end{equation*}
which is a contradiction to \eqref{eq:3.35}. Similarly, the case $t_{q_n}\geq 1+\tau_{q_n}^{-\frac{3}{2}}$ can also be ruled out.

Therefore, \eqref{eq:3.37} holds true, which implies
\begin{equation}\label{eq:3.39}
\Big|\frac{x_0\tau_q}{t_{1,q}}-\tau_qx_0\Big|=|\tau_qx_0|\frac{|1-t_{1,q}|}{t_{1,q}}\leq C\tau_q^{-\frac{1}{2}}\rightarrow 0.
\end{equation}
We then  deduce by Lemma \ref{lem:5.1} that
 \begin{equation}\label{eq:3.40}
 \begin{split}
 &\int_{\mathbb{R}^2}V(x)|w_q^{t_{1,q}}|^2\,dx\\
 &=\int_{\mathbb{R}^2}V(\frac{x}{t_{1,q}\tau_q}+\frac{x_0}{t_{1,q}})\frac{|\varphi_q(x)|^2}{\|\varphi_q\|_2^2}\,dx\\
 &=\tau_q^{-p_{max}}\int_{\mathbb{R}^2}\frac{V\Big(\frac{x}{t_{1,q}\tau_q}+\frac{x_0}{t_{1,q}}\Big)}{\Big|\frac{x}{t_{1,q}\tau_q}+\frac{x_0}{t_{1,q}}-x_0\Big|^{p_{max}}}
 \Big|\frac{x}{t_{1,q}}+\frac{x_0\tau_q}{t_{1,q}}-\tau_qx_0\Big|^{p_{max}}\frac{|\varphi_q(x)|^2}{\|\varphi_q\|_2^2}\,dx\\
 &=\tau_q^{-p_{max}}\|Q\|_2^{-2}\Big(\lambda\int_{\mathbb{R}^2}|x|^{p_{max}}|Q(x)|^2\,dx+o(1)\Big)\\
 &=\Big[\lambda e^{\frac{p_{max}}{2}}\|Q\|_2^{-2}\int_{\mathbb{R}^2}|x|^{p_{max}}|Q(x)|^2\,dx+o(1)\Big]\bigg(\frac{a_q^*}{a}\bigg)^{-\frac{p_{max}}{q-2}}
 \end{split}
 \end{equation}
 as $q\to 2$. We conclude by \eqref{eq:3.31}, \eqref{eq:3.32}, \eqref{eq:3.36} and \eqref{eq:3.40} that
\begin{equation}\label{eq:3.41}
 \max_{0\leq s\leq 1} E_{a,q}(g(s))\leq  \frac{q-2}{2q}\tau_{q}^2+\frac{1}{2}\Big[\lambda e^{\frac{p_{max}}{2}}\|Q\|_2^{-2}\int_{\mathbb{R}^2}|x|^{p_{max}}|Q(x)|^2\,dx+o(1)\Big]\bigg(\frac{a_q^*}{a}\bigg)^{-\frac{p_{max}}{q-2}}.
\end{equation}
By the definition of $c_q$,
 \[
 E_{a,q}(v_q)=c_q\leq \max_{0\leq s\leq 1} E_{a,q}(g(s))=\frac{q-2}{2q}\tau_{q}^2+o(1)
 \]
 as $q\to 2+$. This ends the proof.
\end{proof}

\bigskip

\begin{Remark}\label{remark:3.1}
Let
\begin{equation}\label{cql}
\tilde{c}_q=\frac{q-2}{2q}\tau_{q}^2.
\end{equation}
It follows from \eqref{eq:3.41} and the definition of $c_q$ that
\begin{equation*}
\limsup_{q\to2_+}\frac{c_q-\tilde{c}_q}{\big(\frac{a_q^*}{a}\big)^{-\frac{p_{max}}{q-2}}}\leq \frac{1}{2}\lambda e^{\frac{p_{max}}{2}}\int_{\mathbb{R}^2}\frac{|x|^{p_{max}}|Q(x)|^2}{\|Q\|_2^2}\,dx.
\end{equation*}
\end{Remark}

\bigskip

Now, we estimate the gradient of $v_q$.
\begin{Proposition}\label{Pro4.2} There exists a positive constant $C>0$ such that
\begin{equation}\label{Pro4.2-0}
C(q-2)\tau_q^2\leq \int_{\mathbb{R}^2}|\nabla v_q|^2\,dx\leq \tau_q^2+\frac{C}{q-2},
\end{equation}
as $q\rightarrow 2_+$.
\end{Proposition}
\begin{proof}
Since there exists $\mu_q\in \mathbb{R}$ such that
\begin{equation}\label{Pro4.2-1}
-\Delta v_q+V(x)v_q=av_q^{q+1}+\mu_qv_q,
\end{equation}
and $v_q$ satisfies the Pohozaev identity in  Lemma \ref{lem:5.3}, we deduce from \eqref{eq:2.21} and Proposition \ref{lem:3.2} that
\begin{equation} \label{Pro4.2-2}
\frac{q-2}{2}\int_{\mathbb{R}^2}|\nabla v_q|^2\,dx
+\frac{1}{2}\int_{\mathbb{R}^2}(qV+x\cdot\nabla V)v_q^2\,dx=
\frac{q-2}{2}\tau_q^2+o(1)
\end{equation}
as $q\to2_+$.
The upper bound in \eqref{Pro4.2-0} is then obtained by the assumption (V2).

The lower bound in  \eqref{Pro4.2-0} is obtained indirectly. Indeed, were it not true,  there would exist $\{q_k\}$ with $q_k\to2_+$ as $k\to\infty$ such that
\begin{equation}\label{4.32a}
\int_{\mathbb{R}^2}|\nabla v_{q_k}|^2\,dx=o((q_k-2)\tau_{q_k}^2)
\end{equation}
as $k\to\infty$.
By \eqref{Pro4.2-2} and \eqref{4.32a},
we have
\[
\int_{\mathbb{R}^2}(q_kV+x\cdot\nabla V)v_{q_k}^2\,dx=O((q_k-2)\tau_{q_k}^2).
\]
The fact
\begin{equation}\label{Pro4.2-4}
\frac{p}{2}V-C_1\leq x\cdot \nabla V\leq 2pV+C_2,
\end{equation}
which can be verified as \eqref{eq:3.3aa}, implies
\begin{equation}\label{Pro4.2-5}
\frac{p}{2}\int_{\mathbb{R}^2}Vv_{q_k}^2\,dx-C_1\leq \int_{\mathbb{R}^2}x\cdot \nabla Vv_{q_k}^2\,dx\leq 2p\int_{\mathbb{R}^2}Vv_{q_k}^2\,dx+C_2.
\end{equation}
Therefore,
\begin{equation*}
\int_{\mathbb{R}^2}Vv_{q_k}^2\,dx=O((q_k-2)\tau_{q_k}^2), \ \
\int_{\mathbb{R}^2}x\cdot \nabla Vv_{q_k}^2\,dx=O((q_k-2)\tau_{q_k}^2).
\end{equation*}
Hence, the Pohozaev identity \eqref{eq:5.12} yields that
\[
\int_{\mathbb{R}^2}|\nabla v_{q_k}|^2\,dx\geq \frac{1}{2}\int_{\mathbb{R}^2}x\cdot\nabla Vv_{q_k}^2\,dx=O((q_k-2)\tau_{q_k}^2),
\]
which is a contradiction to \eqref{4.32a}.
\end{proof}

\bigskip

\section{Blow-up analysis}

\bigskip

In this section, using the blow-up argument, we study the asymptotic behavior of $v_q$ as $q\to 2_+$. This is carried through in the proof of  Theorem \ref{thm3}.

\bigskip

{\bf Proof of Theorem \ref{thm3}.} By Proposition \ref{Pro4.2}, we have either the case $(i)$
\begin{equation}\label{thm3-1}
\liminf_{q\rightarrow 2_+}\frac{\int_{\mathbb{R}^2}|\nabla v_q|^2\,dx}{(q-2)\tau_q^2}=C>0,
\end{equation}
or the case $(ii)$
\begin{equation}\label{thm3-2}
\liminf_{q\rightarrow 2_+}\frac{\int_{\mathbb{R}^2}|\nabla v_q|^2\,dx}{(q-2)\tau_q^2}=+\infty.
\end{equation}
We will treat these two cases separately.

\bigskip
In the case $(i)$,  there exists $\{q_k\}$ with $q_k\to2_+$ as $k\to\infty$ such that
\begin{equation}\label{thm3-2a}
\int_{\mathbb{R}^2}|\nabla v_{q_k}|^2\,dx=O(\lambda_{q_k}),
 \end{equation}
where and in the following we denote $\lambda_{q_k}=(q_k-2)\tau_{q_k}^2$. We note $\lambda_{q_k}\to\infty$, as $k\to\infty$. Therefore, we deduce from \eqref{Pro4.2-2}, \eqref{Pro4.2-4} and \eqref{thm3-1} that
\begin{equation}\label{thm3-3}
\int_{\mathbb{R}^2}Vv_{q_k}^2\,dx=O(\lambda_{q_k})
\end{equation}
and
\begin{equation}\label{thm3-3-0}
\int_{\mathbb{R}^2}x\cdot\nabla Vv_{q_k}^2\,dx=O(\lambda_{q_k}).
\end{equation}
By \eqref{eq:2.29} and \eqref{Pro4.2-1}, we have
\begin{equation}\label{thm3-4}
\mu_{q_k}=\int_{\mathbb{R}^2}|\nabla v_{q_k}|^2\,dx+\int_{\mathbb{R}^2}V(x)v_{q_k}^2\,dx-a\int_{\mathbb{R}^2}|v_{q_k}|^{q+2}\,dx< \lambda_1.
\end{equation}
Hence, by \eqref{thm3-3} and \eqref{thm3-4},
$$
a\int_{\mathbb{R}^2}|v_{q_k}|^{q+2}\,dx\geq \int_{\mathbb{R}^2}V(x)v_{q_k}^2\,dx-\lambda_1=O(\lambda_{q_k}).
$$
On the other hand, by the Pohozaev identity \eqref{eq:5.12}, we have
\begin{equation}\label{5.6a}
\int_{\mathbb{R}^2}|\nabla v_{q_k}|^2\,dx-\frac{1}{2}\int_{\mathbb{R}^2}x\cdot\nabla V|v_{q_k}|^2\,dx=\frac{aq_k}{q_k+2}\int_{\mathbb{R}^2}|v_{q_k}|^{q_k+2}\,dx.
\end{equation}
Thus, we obtain from \eqref{thm3-3-0}
for $k$ large enough that
$$
\frac{q_ka}{q_k+2}\int_{\mathbb{R}^2}|v_{q_k}|^{q+2}\,dx\leq \int_{\mathbb{R}^2}|\nabla v_{q_k}|^2\,dx=O(\lambda_{q_k}).
$$
Consequently,
\begin{equation}\label{thm3-5}
a\int_{\mathbb{R}^2}|v_{q_k}|^{q+2}\,dx=O(\lambda_{q_k}).
\end{equation}
By \eqref{thm3-2a}, there exist $C_1>0$ and $C_2>0$ such that
\begin{equation}\label{thm3-4ab}
C_1(q_k-2)\Big(\frac{2a_{q_k}^*}{q_ka}\Big)^{\frac{2}{q_k-2}}\leq
\int_{\mathbb{R}^2}|\nabla v_{q_k}|^2\,dx
\leq C_2(q_k-2)\Big(\frac{2a_{q_k}^*}{q_ka}\Big)^{\frac{2}{q_k-2}}.
\end{equation}
Let
\begin{equation} \label{thm3-6a}
\varepsilon_q=\|\nabla v_q\|_2^{-1}
\end{equation}
 and
\begin{equation}\label{thm3-7a}
\tilde{w}_q(x)=\varepsilon_qv_q(\varepsilon_q x).
\end{equation}
Then,
\begin{equation}\label{thm3-7}
\|\nabla\tilde{w}_q \|_2^2=\|\tilde{w}_q\|_2^2=1.
\end{equation}
It is known from  Lemma \ref{lem:5.1} that $a_q^*\to a^*$ as $q\to2_+$,  then by \eqref{thm3-4ab},
\begin{equation}\label{thm3-6bb}
\varepsilon_{q_k}^{q_k-2}=\Big(\int_{\mathbb{R}^2}|\nabla v_{q_k}|^2\,dx\Big)^{\frac{2-q_k}{2}}\rightarrow \frac{a}{a^*}.
\end{equation}
Moreover, by  \eqref{thm3-2a}, \eqref{thm3-5} , \eqref{thm3-6a} and \eqref{thm3-6bb}, there exist $C_3>0$ and $C_4>0$ such that
\begin{equation}\label{thm3-8}
C_3\leq \int_{\mathbb{R}^2}|\tilde{w}_{q_k}|^{{q_k}+2}\,dx=\varepsilon_{q_k}^{{q_k}-2}
\frac{\int_{\mathbb{R}^2}|v_{q_k}|^{{q_k}+2}\,dx}{\int_{\mathbb{R}^2}|\nabla v_{q_k}|^2\,dx}
\leq C_4.
\end{equation}
Arguing as \eqref{3.29kkkk}, we find from \eqref{thm3-7} and \eqref{thm3-8} that there exist $\{y_{q_k}\}\subset\mathbb{R}^2$, $R_0>0$ and $\eta>0$ such that
\begin{equation}\label{thm3-9}
\liminf_{k\to\infty}\int_{B_{R_0}(y_{q_k})}|\tilde{w}_{q_k}|^2\,dx\geq \eta.
\end{equation}

Denote $w_{q_k}=\tilde{w}_{q_k}(x+y_{q_k})=\varepsilon_{q_k}v_{q_k}(\varepsilon_{q_k}(x+y_{q_k}))$. We have
\begin{equation}\label{thm3-10}
\begin{split}
&\|\nabla w_{q_k}\|_2^2=\|w_{q_k}\|_2^2=1;\\
&C_3\leq \int_{\mathbb{R}^2}|w_{q_k}|^{{q_k}+2}\,dx\leq C_4;\\
&\liminf_{k\to\infty}\int_{B_{R_0}(0)}|w_{q_k}|^2\,dx\geq \eta.
\end{split}
\end{equation}
Hence, there exist a sequence $\{q_k\}$ with $q_k\to 2_+$ as $k\to\infty$ and $w\in H^1(\mathbb{R}^2)$ such that $w_{q_k}\rightharpoonup w$ weakly in $H^1(\mathbb{R}^2)$ and $w_{q_k}\to w\neq0$ strongly in
$L_{loc}^\gamma(\mathbb{R}^2)$ for any $\gamma\geq 2$.
By \eqref{Pro4.2-1}, we see that $w_{q_k}$ solves
\begin{equation}\label{thm3-11}
\begin{split}
-\Delta w_{q_k}+\varepsilon_{q_k}^2V(\varepsilon_{q_k}(x+y_{q_k}))w_{q_k}
=\varepsilon_{q_k}^2\mu_{q_k}w_{q_k}+\varepsilon_{q_k}^{2-q}aw_{q_k}^{{q_k}+1}.
\end{split}
\end{equation}

Now, we study the asymptotic behavior of $w_{q_k}$ and the limiting equation of \eqref{thm3-11}.

Let us consider the limiting behavior of  $\varepsilon_{q_k}^2\mu_{q_k}$ first. By \eqref{thm3-4}, we have
\[
\varepsilon_{q_k}^2\mu_{q_k}\leq \varepsilon_{q_k}^2\lambda_1\to0,\ \ {\rm as} \ \ q\to2.
\]
On the other hand, equations \eqref{thm3-2a}, \eqref{thm3-4}, \eqref{thm3-5} and \eqref{thm3-6a} yield
\[
\varepsilon_{q_k}^2\mu_{q_k}\geq -\varepsilon_{q_k}^2a\int_{\mathbb{R}^2}|v_{q_k}|^{q_k+2}\,dx\geq -C\varepsilon_{q_k}^2\lambda_{q_k}\geq -C.
\]
So we may assume
\begin{equation}\label{thm3-12}
\varepsilon_{q_k}^2\mu_{q_k}\to -\beta^2\leq 0.
\end{equation}

\bigskip

Next, we study the limiting behavior of  $\varepsilon_{q_k}^2V(\varepsilon_{q_k}(x+y_{q_k}))$.

If $\liminf_{k\to\infty}|\varepsilon_{q_k}y_{q_k}|\leq C$  with $C>0$ independent of $k$, then there exists $\{q_k\}$ such that $|\varepsilon_{q_k}y_{q_k}|\leq C$. So for any $M>0$, we have
\begin{equation}\label{thm3-13}
\varepsilon_{q_k}^2V(\varepsilon_{q_k}(x+y_{q_k}))\leq \varepsilon_{q_k}^2
(|\varepsilon_{q_k}x|+|\varepsilon_{q_k}y_{q_k}|+\max_i|x_i|)^p\to0
\end{equation}
uniformly for  $x\in B_M(0)$ as  $k\to \infty$.

On the other hand, if $\liminf_{k\to\infty}|\varepsilon_{q_k}y_{q_k}|=+\infty$, we may find a sequence $\{q_k\}$ with $q_k\to 2_+$ as $k\to\infty$ such that $|\varepsilon_{q_k}y_{q_k}|\to +\infty$, as $k\to\infty$. For any
$M>0$, if $|x|\leq M$ and $k$ is large enough,  we obtain
\begin{equation}\label{thm3-14}
V(\varepsilon_{q_k}(x+y_{q_k}))=\prod_{i=1}^n|\varepsilon_{q_k}x-x_i+\varepsilon_{q_k}y_{q_k}|^{p_i}
\geq \frac{1}{2^p}\prod_{i=1}^n|-x_i+\varepsilon_{q_k}y_{q_k}|^{p_i}.
\end{equation}
By \eqref{thm3-3},
\begin{equation*}
\begin{split}
\int_{B_{R_0}(0)}V(\varepsilon_{q_k}(x+y_{q_k}))w_{q_k}^2(x)\,dx&\leq \int_{\mathbb{R}^2}V(\varepsilon_{q_k}(x+y_{q_k}))w_{q_k}^2(x)\,dx\\
&=\int_{\mathbb{R}^2}V(x)v_{q_k}^2(x)\,dx\\
&\leq C\lambda_{q_k}
\end{split}
\end{equation*}
Hence,  for $k$ large enough,
\[
\int_{B_{R_0}(0)}\lambda_{q_k}^{-1}\prod_{i=1}^n|-x_i+\varepsilon_{q_k}y_{q_k}|^{p_i}w_{q_k}^2(x)\,dx\leq C.
\]
By \eqref{thm3-9},
\begin{equation}\label{thm3-15a}
\lambda_{q_k}^{-1}\prod_{i=1}^n|-x_i+\varepsilon_{q_k}y_{q_k}|^{p_i}\leq \frac{C}{\eta}.
 \end{equation}
 Note that  $\varepsilon_{q_k}^{-2}=O(\lambda_{q_k})$ via  \eqref{thm3-2a} and \eqref{thm3-6a}, we have up to a subsequence that
\begin{equation}\label{thm3-15}
\varepsilon_{q_k}^{2}\prod_{i=1}^n|-x_i+\varepsilon_{q_k}y_{q_k}|^{p_i}\to \mu\geq 0.
\end{equation}
It is readily to verify that for any $M>0$ and $x\in B_M(0)$,
\[
\lim_{k\to\infty}\varepsilon_{q_k}^{2}V(\varepsilon_{q_k}(x+y_{q_k}))
=\lim_{k\to\infty}\varepsilon_{q_k}^{2}\prod_{i=1}^n|-x_i+\varepsilon_{q_k}y_{q_k}|^{p_i}.
\]
Hence, \eqref{thm3-15} implies that either
\begin{equation}\label{thm3-16}
\varepsilon_{q_k}^{2}V(\varepsilon_{q_k}(x+y_{q_k}))\to \mu>0
\end{equation}
or
\begin{equation}\label{thm3-17}
\varepsilon_{q_k}^{2}V(\varepsilon_{q_k}(x+y_{q_k}))\to 0
\end{equation}
uniformly in  $x\in B_M(0)$.

In summary of \eqref{thm3-12},\eqref{thm3-13}, \eqref{thm3-16} and \eqref{thm3-17}, in the case $(i)$ we have  the following subcases:

Subcase $(i)$:\ \ \ \ $\varepsilon_{q_k}^2\mu_{q_k}\to 0$, \, and \,\,  \eqref{thm3-16} holds;

Subcase $(ii)$: \ $\varepsilon_{q_k}^2\mu_{q_k}\to 0$, \, and \, \, \eqref{thm3-17} holds;

Subcase $(iii)$: $\varepsilon_{q_k}^2\mu_{q_k}\to -\beta^2<0$, \, and  \,\, \eqref{thm3-16} holds;

Subcase $(iv)$: $\varepsilon_{q_k}^2\mu_{q_k}\to -\beta^2<0$, \, and  \,\, \eqref{thm3-17} holds.
\bigskip

Taking the  limit $k\to\infty$ in  \eqref{thm3-11}, we obtain that $w$ satisfies correspondingly in the subcase $(i)$ that
\begin{equation}\label{thm3-18}
-\Delta w+\mu w=a^*w^3;
\end{equation}
in the subcase $(ii)$ that
\begin{equation}\label{thm3-19}
-\Delta w=a^*w^3;
\end{equation}
in the subcase $(iii)$ that
\begin{equation}\label{thm3-20}
-\Delta w+(\mu+\beta^2) w=a^*w^3;
\end{equation}
and in the subcase $(iv)$ that
\begin{equation}\label{thm3-21}
-\Delta w+\beta^2 w=a^*w^3.
\end{equation}

In the subcase $(i)$, by the uniqueness of positive solution of \eqref{thm3-18} and $w\neq0$, there exists $y_0\in \mathbb{R}^2$ such that
$$
w=\frac{\sqrt{\mu}}{\|Q\|_2}Q(\sqrt{\mu}(x-y_0)),
$$
which implies $\|w\|_2^2=1.$
Hence, $w_{q_k}\to w$ strongly in $L^2(\mathbb{R}^2)$, that is
\begin{equation}\label{thm3-22}
\varepsilon_{q_k}v_{q_k}(\varepsilon_{q_k}(x+y_{q_k}))\to \frac{\sqrt{\mu}}{\|Q\|_2}Q(\sqrt{\mu}(x-y_0)) \ \ \ {\rm strongly \ \ in} \ \
L^2(\mathbb{R}^2).
\end{equation}

The subcase $(ii)$ can not happen. Indeed, if it would happen on the contrary,  the fact that $w\geq 0$ and $w\neq0$ would imply that \eqref{thm3-19} admits a positive solution, which contradicts the Liouville type theorem.

In the same way to drive  \eqref{thm3-22}, we can show in the subcase $(iii)$ that, there exists $y_1\in \mathbb{R}^2$ such that
\begin{equation}\label{thm3-23}
\varepsilon_{q_k}v_{q_k}(\varepsilon_{q_k}(x+y_{q_k}))\to \frac{\sqrt{\mu+\beta^2}}{\|Q\|_2}Q(\sqrt{\mu+\beta^2}(x-y_1)) \ \ \ {\rm strongly \ \ in} \ \
L^2(\mathbb{R}^2),
\end{equation}
and in the subcase $(iv)$ that, there exists $y_2\in \mathbb{R}^2$ such that
\begin{equation}\label{thm3-24}
\varepsilon_{q_k}v_{q_k}(\varepsilon_{q_k}(x+y_{q_k}))\to \frac{\beta}{\|Q\|_2}Q(\beta(x-y_2)) \ \ \ {\rm strongly \ \ in} \ \
L^2(\mathbb{R}^2).
\end{equation}
Therefore, the conclusions in Theorem \ref{thm3} are valid for the case $(i)$.

\bigskip

Now, we turn to the case $(ii)$.

By \eqref{thm3-2}, there exists $\{q_k\}$ such that
\begin{equation}\label{thm3-2aaa}
\int_{\mathbb{R}^2}|\nabla v_{q_k}|^2\,dx>>O(\lambda_{q_k}).
 \end{equation}
Then, equation \eqref{Pro4.2-2} and \eqref{Pro4.2-4} imply
\begin{equation}\label{thm3-25}
\int_{\mathbb{R}^2}Vv_{q_k}^2\,dx\leq C\lambda_{q_k},
\end{equation}
which yields via \eqref{Pro4.2-4} that
\begin{equation}\label{thm3-26}
\Big|\int_{\mathbb{R}^2}x\cdot \nabla Vv_{q_k}^2\,dx\Big|\leq C\lambda_{q_k}.
\end{equation}
By the Pohozaev identity \eqref{5.6a} and \eqref{Pro4.2-4},
\begin{equation}\label{thm3-27}
\lim_{k\rightarrow\infty}\frac{\int_{\mathbb{R}^2}|\nabla v_{q_k}|^2\,dx} { \frac{a}{2}\int_{\mathbb{R}^2}\big|v_{q_k}\big|^{q_k+2}\,dx}=1.
\end{equation}

Now, we apply the blow-up analysis for $\tilde{w}_{q_k}$, which is defined in \eqref{thm3-7a}.

By \eqref{Pro4.2-0}, \eqref{thm3-6a}and \eqref{thm3-27},
\[
\varepsilon_{q_k}^{q_k-2}=\Big(\int_{\mathbb{R}^2}|\nabla v_{q_k}|^2\,dx\Big)^{\frac{2-q}{2}}
\geq \big(2\tau_{q_k}^2)^{\frac{q-2}{2}}=2^{\frac{q-2}{2}}\frac{qa}{2a_q^*}\to \frac{a}{a^*}
\]
and
\[
\varepsilon_{q_k}^{q_k-2}\leq \big(C\lambda_{q_k}\big)^{\frac{2-q}{2}}=C^{\frac{q-2}{2}}(q-2)^{\frac{q-2}{2}}\frac{qa}{2a_q^*}
\to \frac{a}{a^*}.
\]
Thus, we have
\begin{equation}\label{thm3-28ab}
\varepsilon_{q_k}^{q_k-2}\to \frac{a}{a^*}.
\end{equation}
By \eqref{thm3-27},  there exist $C_1>0$ and $C_2>0$ such that
 \begin{equation}\label{thm3-28}
 C_1\leq \int_{\mathbb{R}^2}|\tilde{w}_{q_k}|^{{q_k}+2}\,dx=\varepsilon_{q_k}^{{q_k}-2}
\frac{\int_{\mathbb{R}^2}|v_{q_k}|^{{q_k}+2}\,dx}{\int_{\mathbb{R}^2}|\nabla v_{q_k}|^2\,dx}
\leq C_2.
 \end{equation}
 Proceeding as the case $(i)$, there exists $y_{q_k}\in \mathbb{R}^2$ such that  \eqref{thm3-9} holds. Let
 $$
 w_{q_k}=\tilde{w}_{q_k}(x+y_{q_k})=\varepsilon_{q_k}v_q(\varepsilon_{q_k}(x+y_{q_k})).
 $$
 Hence, there is a subsequent of $\{w_{q_k}\}$, still denoted by $\{w_{q_k}\}$, such that $w_{q_k}\rightharpoonup w\neq0$ weakly in $H^1(\mathbb{R}^2)$ and
 $w_{q_k}\to w$  strongly in $L_{loc}^\gamma(\mathbb{R}^2)$ for any $\gamma\geq2$. Moreover,
 $w_q$ solves \eqref{thm3-11}.

In the same way as the case $(i)$, we analyze the limiting behavior of $\varepsilon_{q_k}^2V(\varepsilon_{q_k}(x+y_{q_k}))$.

If $\liminf_{k\to\infty}|\varepsilon_{q_k}y_{q_k}|\leq C$ with $C>0$ independent of $k$, there exists $\{q_k\}$ such that $|\varepsilon_{q_k}y_{q_k}|\leq C$ and \eqref{thm3-16} holds.

If $\liminf_{k\to\infty}|\varepsilon_{q_k}y_{q_k}|=\infty$, then there exists $\{q_k\}$ such that $|\varepsilon_{q_k}y_{q_k}|\to\infty$. By \eqref{thm3-25}, we proceed as the  case $(i)$ that \eqref{thm3-15a} also holds. Hence, for any $x\in B_M(0)$, due to \eqref{thm3-6a} and \eqref{thm3-2aaa}, we find
\begin{equation*}
\begin{split}
\varepsilon_{q_k}^2V(\varepsilon_{q_k}(x+y_{q_k}))
&\leq \varepsilon_{q_k}^2\prod_{i=1}^n
\Big(|\varepsilon_{q_k}x|+|-x_i+\varepsilon_{q_k}y_{q_k}|\Big)^{p_i}\\
&\leq 2\varepsilon_{q_k}^2\prod_{i=1}^n|-x_i+\varepsilon_{q_k}y_{q_k}|^{p_i}\\
&\leq 2\varepsilon_{q_k}^2 \lambda_{q_k}\lambda_{q_k}^{-1}\prod_{i=1}^n|-x_i+\varepsilon_{q_k}y_{q_k}|^{p_i}\\
&\leq \frac{C\varepsilon_{q_k}^2 \lambda_{q_k}}{\eta}\rightarrow 0.
\end{split}
\end{equation*}

Now, we treat the factor $\varepsilon_{q_k}^2\mu_{q_k}$. By \eqref{thm3-4},  \eqref{thm3-27}, \eqref{thm3-2aaa} and  \eqref{thm3-25}, we have  $-\mu_{q_k}=O(\varepsilon_{q_k}^{-2})$. Hence, there exists $\{q_k\}$ and $\beta>0$  such that $\varepsilon_{q_k}^2\mu_{q_k}\to -\beta^2$.  Taking the limit in \eqref{thm3-11}, we see that $w$ satisfies
\[
-\Delta w+\beta^2w=a^*w^3.
\]
 Similar to the proof in the subcase $(i)$,  there exists $y_0\in\mathbb{R}^2$ and $\{v_{q_k}\}$ such that
 \begin{equation}\label{thm3-29}
 \varepsilon_{q_k}v_{q_k}(\varepsilon_{q_k}(x+y_{q_k}))\to \frac{\beta}{\|Q\|_2}Q(\beta(x-y_0)) \ \ \ {\rm strongly \ in} \ \
L^2(\mathbb{R}^2).
 \end{equation}
Equation \eqref{eq:1.13} then follows from \eqref{thm3-22}--\eqref{thm3-24} and \eqref{thm3-29}. The proof is complete.\qquad$\Box$

\bigskip

Finally, we  deal with the special case: $V(x)=|x|^p$ with $p\geq 1$.
\bigskip

{\bf Proof of Corollary \ref{cor2}.}
 In the case  $V(x)=|x|^p$, $v_q$ satisfies
 \begin{equation}\label{thm3-30}
 -\Delta v_q+|x|^pv_q=\mu_qv_q+v_q^{q+1}.
 \end{equation}
By the classical bootstrap argument, we have $v_q\in C^2(\mathbb{R}^2)$ and $\lim_{|x|\to \infty}v_q(x)=0$. We know from Theorem 2 in \cite{LN} that $v_q$ is radially symmetric and decreasing from the origin. Since the Sobolev inclusion $H^1_r(\mathbb{R}^2)\hookrightarrow L^\gamma(\mathbb{R}^2)$ is compact for any $\gamma>2$, we may choose $y_q=0$ in \eqref{thm3-9}, and it is readily to show that
 \eqref{thm3-17} holds true.

 We may verify through the proof of Theorem \ref{thm3} that only subcase $(iv)$ and case $(ii)$ may happen, that is,
 $$
 \varepsilon_{q_k}^2\mu_{q_k}\to -\beta^2<0
 $$
 and
 $$
 \varepsilon_{q_k}^{2-q}a\to a^*.
 $$
By \eqref{thm3-11}, we have
  \begin{equation}\label{dnmd}
  -\Delta w_{q_k}\leq \Big(\varepsilon_q^{2-q}aw_{q_k}^{q_k}\Big)w_{q_k}.
  \end{equation}
  Using the De Diorgi-Nash-Moser estimate, see Theorem 4.1 in \cite{HL}, we obtain
  \[
  \max_{B_1(\xi)}w_{q_k}\leq C\Big(\int_{B_2(\xi)}|w_{q_k}|^\gamma\,dx\Big)^{\frac{1}{\gamma}}
  \]
  for any $\gamma\geq 2$ and $\xi\in\mathbb{R}^2$. The constant $C$ depends only on the bound of $\|w_{q_k}\|_{L^\gamma(B_2(\xi))}$. It follows from $\|w_q\|_{H^1(\mathbb{R}^2)}^2=1$, $H^1(\mathbb{R}^2)\hookrightarrow L^{\gamma}(\mathbb{R}^2)$ for any $\gamma\geq 2$ and \eqref{dnmd} that $w_{q_k}$ is uniformly bounded in $L^\infty(\mathbb{R}^2)$.

  By Lemma 1.7.3 in \cite{CA} and $\|w_{q_k}\|_2^2=1$,
  \[
  w_q(x)\leq \frac{C}{|x|}.
  \]
  So there exists $R>0$, independent of $k$, and $k$ large enough, for $|x|\geq R$ that,
 \[
 -\Delta w_{q_k}\leq -\frac{1}{3}\beta^2w_{q_k}.
 \]
  By the comparison principle,
 \[
 w_{q_k}(x)\leq Ce^{-\frac{1}{4}\beta^2|x|}
 \]
 for $|x|\geq R$. Since $w_{q_k}$ is uniformly bounded in $L^\infty(\mathbb{R}^2)$, we have
\begin{equation}\label{thm3-31}
w_{q_k}(x)\leq Ce^{-\frac{1}{4}\beta^2|x|}.
 \end{equation}
 for any $x\in\mathbb{R}^2$, where $C$ is independent of $k$.
 Therefore,
 \begin{equation}\label{thm3-32}
 \begin{split}
 \int_{\mathbb{R}^2}V(x)v_{q_k}^2\,dx=\int_{\mathbb{R}^2}V(\varepsilon_{q_k}x)w_{q_k}^2\,dx
 \leq C\int_{\mathbb{R}^2}\varepsilon_{q_k}^p|x|^pe^{-\frac{1}{2}\beta^2|x|}\,dx\to0.
 \end{split}
\end{equation}
The inequalities \eqref{Pro4.2-4}  and \eqref{thm3-32} yield
\begin{equation}\label{thm3-33}
\Big|\int_{\mathbb{R}^2}x\cdot\nabla Vv_{q_k}^2\,dx\Big|\leq C.
\end{equation}
We conclude from \eqref{Pro4.2-2}, \eqref{thm3-32} and \eqref{thm3-33}  that
\begin{equation}\label{thm3-34}
\lim_{k\to \infty}\tau_{q_k}^{-2}\int_{\mathbb{R}^2}|\nabla v_{q_k}|^2\,dx=1,
\end{equation}
and case $(i)$ can not happen.

By \eqref{5.6a}, \eqref{thm3-33} and \eqref{thm3-34}, we obtain
\begin{equation}\label{zzh1}
\lim_{k\rightarrow\infty}\frac{\int_{\mathbb{R}^2}|\nabla v_{q_k}|^2\,dx} { \frac{a}{2}\int_{\mathbb{R}^2}\big|v_{q_k}\big|^{q_k+2}\,dx}=1.
\end{equation}
It follows from \eqref{thm3-4} and \eqref{thm3-32}--\eqref{zzh1} that  $\lim_{q_k\to\infty}\frac{\mu_{q_k}}{\tau_{q_k}^2}=-1$, and then
$\varepsilon_{q_k}^2\mu_{q_k}\to -1$.  By \eqref{thm3-12},  $\beta^2=1$. The proof is complete.  \qquad$\Box$

\bigskip
\vspace{2mm}
\noindent{\bf Acknowledgment} The first author is supported by NNSF of China, No:11671179 and 11771300. The second author is supported by NNSF of China, No:11701260.

{\small
\end{document}